\documentclass{amsart}
\usepackage{xcolor}
\usepackage[utf8]{inputenc}
\usepackage[all]{xy}
\usepackage{enumerate,multicol,stmaryrd}
\usepackage{tabularx}
\usepackage{amsaddr}

 \usepackage{amsbsy}			% para símbolos matemáticos em negrito
 \usepackage{amscd}			% para diagramas
 \usepackage{amsfonts}			% fontes AMS
 \usepackage{amsmath}			% facilidades matemáticas
 \usepackage{amssymb}			% para os símbolos mais antigos
 \usepackage{amstext}			% para fragmentos tipo texto em modo matemático
\usepackage{amsthm}			% para teoremas
\usepackage{hyperref}			% Amplo suporte para hipertexto em LaTeX
\usepackage{cleveref}			% Referência cruzada inteligente
\usepackage{dsfont}			% para o estilo de conjuntos de números $\mathds{R}$
 \usepackage{ifthen}			% comandos de condição em LaTeX
\usepackage{listings}           	% para inserir códigos de outras linguagens de programação
 \usepackage{lscape}             	% para imprimir alguma página no formato paisagem
\usepackage{mathabx}			% conjunto de simbolos matemáticos
 \usepackage{mathrsfs}			% suporte para fontes RSFS
 \usepackage{pdfpages}           	% para inserir páginas PDF no texto

\usepackage{graphicx}
\usepackage{graphics}      % para acrescentar links
\usepackage{fancybox}

\newenvironment{declaration}{\noindent\textbf{Declaration and Statements}}{}
\newtheorem{theorem}{Theorem}[section]

\newtheorem{proposition}[theorem]{Proposition}
\newtheorem{definition}{Definition}
\newtheorem{corollary}[theorem]{Corollary}
\newtheorem{lemma}{Lemma}
\newtheorem{remark}{Remark}
\newtheorem{example}{Example}

\newtheorem{ithm}{Theorem}[section]

\def\cal#1{\mathcal{#1}}
\def\bb#1{\mathbb{#1}}
\def\lie#1{\mathfrak{#1}}

\usepackage[all]{xy}

\def\co{\colon}

\newcommand{\ga}{\textsl{g}}

\makeatletter
\renewcommand{\email}[2][]{%
	\ifx\emails\@empty\relax\else{\g@addto@macro\emails{,\space}}\fi%
	\@ifnotempty{#1}{\g@addto@macro\emails{\textrm{(#1)}\space}}%
	\g@addto@macro\emails{#2}%
}
\makeatother

\title{Symmetric spaces as adjoint orbits and their geometries}

\author{Leonardo F. Cavenaghi$^{\dagger}$ }
\address{Instituto de Matemática, Estatística e Computação Científica -- Unicamp, Rua Sérgio Buarque de Holanda, 651, 13083-859, Campinas, SP, Brazil}
\email[$\dagger$]{leonardofcavenaghi@gmail.com}
\author{Carolina Garcia$^{\ast}$}
\address{Instituto de Matemática, Estatística e Computação Científica -- Unicamp, Rua Sérgio Buarque de Holanda, 651, 13083-859, Campinas, SP, Brazil}
\email[$\ast$]{acgr93@gmail.com}
\author{Lino Grama$^{\star}$}
\address{Instituto de Matemática, Estatística e Computação Científica -- Unicamp, Rua Sérgio Buarque de Holanda, 651, 13083-859, Campinas, SP, Brazil}
\email[$\star$]{lino@ime.unicamp.br}
\author{Luiz A. B. San Martin$^{\bullet}$}
\address{Instituto de Matemática, Estatística e Computação Científica -- Unicamp, Rua Sérgio Buarque de Holanda, 651, 13083-859, Campinas, SP, Brazil}
\email[$\bullet$]{smartin@unicamp.br}

\begin{document}

\subjclass[2000]{53C30,53C35,53D12}
\keywords{Symmetric spaces; Adjoint Orbits; Flag manifolds; Lagrangian submanifolds; non-negative curvatures}
\pretolerance=10000

\begin{abstract}
We realize specific classical symmetric spaces, like the semi-K\"ahler symmetric spaces discovered by Berger, as cotangent bundles of symmetric flag manifolds. These realizations enable us to describe these cotangent bundles' geodesics and Lagrangian submanifolds.

As a final application, we present the first examples of vector bundles over simply connected manifolds with nonnegative curvature that cannot accommodate metrics with nonnegative sectional curvature, even though their associated unit sphere bundles can indeed accommodate such metrics. Our examples are derived from explicit bundle constructions over symmetric flag spaces.
\end{abstract}

\maketitle

\section*{Declaration}
\begin{declaration}
All authors declare that they have no conflicts of interest.

Our manuscript has no associated data.
\end{declaration}

\section{Introduction}

Several authors, including Berger, Kobayashi, Nomizu, and Helgason, have extensively investigated symmetric spaces. As per the definition in \cite{KN}, a symmetric space is a triple denoted as $(G, H, \sigma)$, comprising a connected Lie group $G$, a closed subgroup $H$, and an involutive automorphism $\sigma$ of $G$ with the property that $G_0^\sigma \subset H \subset G^\sigma$. Here, $G^\sigma$ represents the set of elements in $G$ for which $\sigma(g) = g$, and $G_0^\sigma$ stands for the identity-connected component of $G^\sigma$. When the triple $(G, H, \sigma)$ fulfills these conditions, the quotient space $G/H$ is termed an affine symmetric manifold.

Let $\mathfrak{g}$ and $\mathfrak{h}$ be the Lie algebras of $G$ and $H$, respectively, then the symmetric Lie algebra $( \mathfrak{g}, \mathfrak{h}, \sigma)$ has the canonical decomposition $ \mathfrak{g} = \mathfrak{h}+\mathfrak{m}$ satisfying $[ \mathfrak{h}, \mathfrak{m} ] \subset \mathfrak{m}$ and $[ \mathfrak{m}, \mathfrak{m} ] \subset \mathfrak{h}$. If $\mathrm{ad}_\mathfrak{g} (\mathfrak{h})$ is compact, then $(\mathfrak{g},\mathfrak{h},\sigma)$ is called an \emph{orthogonal symmetric Lie algebra}. These are compact or non-compact types, and four classes with their geometric properties obstructed. See \cite[Theorem 8.6, p. 256]{KN}
\begin{theorem} 
	 Let $(G,H,\sigma)$ be a symmetric space with $\mathrm{Ad}_G(\lie h)$ compact and let $(\mathfrak{g},\mathfrak{h},\sigma)$ be its orthogonal symmetric Lie algebra. Take any $G$-invariant Riemannian metric on $G/H$. Then,
	\begin{itemize}
		\item[(1)] if $(\mathfrak{g},\mathfrak{h},\sigma)$ is of compact type, then $G/H$ is a compact Riemannian symmetric space with nonnegative sectional curvature and positive definite Ricci tensor
		\item[(2)] if $(\mathfrak{g},\mathfrak{h},\sigma)$ is of non-compact type, then $G/H$ is a simply connected non-compact Riemannian symmetric space with non-positive sectional curvature and negative definite Ricci tensor and is diffeomorphic to a Euclidean space.
	\end{itemize}
\end{theorem}

In \cite{HELGA}, Helgason presents a table featuring Riemannian symmetric spaces, including:
\begin{enumerate}
    \item The Grassmannian $\mathrm{SU}(p+q)/\mathrm{S}(\mathrm{U}(p)\times \mathrm{U}(q))$: This represents complex subspaces of $\mathbb{C}^n$.

    \item The space $\mathrm{Sp}(n)/\mathrm{U}(n)$: This space pertains to complex structures on $\mathbb{H}^n$ that are compatible with its standard inner product.

    \item The space $\mathrm{SO}(2n)/\mathrm{U}(n)$: This space is associated with orthogonal complex structures on $\mathbb{R}^{2n}$. 
\end{enumerate}
All of these examples exhibit orthogonal symmetric Lie algebras.

Let $G$ be a complex semisimple Lie group, with Lie algebra $\mathfrak{g}$. Denote by $\Sigma$ the simple roots of $\mathfrak{g}$ concerning a Cartan subalgebra $\mathfrak{h}$. For each subset $\Theta\subset\Sigma$ one can define a parabolic subgroup $P_\Theta$ and the corresponding flag manifold is the homogeneous space $\mathbb{F}_\Theta=G/P_\Theta$. Let $U\subset G$ be the compact real form. One can show that $U$ also acts transitively on $\mathbb{F}_\Theta$ and we have $\mathbb{F}_\Theta=U/U_\Theta$, where $U_\Theta=U\cap P_\Theta$. When there is no risk of confusion, we omit the $\Theta$ in the notation. See Section \ref{sec:roots} for further details.  

In this paper, we will explore intriguing relationships between the geometry of symmetric flag manifolds and their cotangent bundles. Let us summarize its content: 
\begin{itemize}
	\item we prove that if a complex flag manifold $U/U_\Theta$ is a symmetric space, then its cotangent bundle is a symmetric space (Theorem \ref{thm:mainintro});
	\item we determine an automorphism $\sigma$ of $G$ such that $(U,U_\Theta,\sigma)$ and $(G,G^\sigma ,\sigma)$ are symmetric spaces with $G/G^\sigma$ being a realization of the cotangent bundle of $U/U_\Theta$ (Theorem \ref{thm:mainintro});
	\item we find Lagrangian submanifolds of the adjoint orbit $G/G^\sigma$ with respect to a non-canonical symplectic form (Theorem \ref{thm:secondmainintro});
 \item we produce the first example of a vector bundle over a simply connected Riemannian manifold whose associated sphere bundle admits a metric of nonnegative sectional curvature, although the corresponding vector bundle does not (Theorem \ref{thm:ziller}).
\end{itemize}

It is worth noting that a crucial element in constructing our main results is the description of the adjoint orbit of a complex semisimple (non-compact) Lie group $G$ as the cotangent bundle of a complex flag manifold, as outlined in \cite{ADORB}. More precisely, recall that given $\Theta \subset \Sigma$, an element $H_{\Theta} \in \mathfrak{h}$ is said to be \textit{characteristic} for $\Theta$ if one recognizes $\Theta = \{ \alpha \in \Sigma : \alpha (H_\Theta) =0 \}$. 

 \begin{theorem}
    \label{bbb}
	\cite[Theorem 2.1]{ADORB} The adjoint orbit $\mathcal{O}(H_\Theta) =\mathrm{Ad}(G)\cdot H_\Theta \approx G/Z_\Theta$ of the characteristic element $H_\Theta$ is a $C^\infty$ vector bundle over $\mathbb{F}_\Theta$ that is isomorphic to the cotangent bundle $T^*\mathbb{F}_\Theta$. Moreover, we can write down a diffeomorphism $\iota :\mathrm{Ad}(G)\cdot H_\Theta \rightarrow T^*\mathbb{F}_\Theta$ such that
	\begin{enumerate}[$(a)$]
		\item $\iota$ is equivariant with respect to the actions of $U$, that is, for all $u\in U$,
		\begin{equation*}
			\iota \circ \mathrm{Ad} (u) = \tilde{u} \circ \iota,
		\end{equation*} 
		where $\tilde{u}$ is the lifting to $T^* \mathbb{F}_\Theta$ (via the differential) of the action of $u$ on $\mathbb{F}_\Theta$; and
		\item the pullback of the canonical simplectic form on $T^*\mathbb{F}_\Theta$ by $\iota$ is the (real) Kostant-Kirillov-Souriau form on the orbit.
	\end{enumerate}
\end{theorem}

\

Next, we will provide a detailed discussion of our results.

\begin{ithm}\label{thm:mainintro}
Let $\mathbb{F}_\Theta$ be a symmetric flag space identified with the homogeneous space $\mathbb{F}_\Theta = U/U_\Theta$.
Then there exists an element $H_\Theta$ in a fixed Cartan sub-algebra of the compact Lie algebra $\mathfrak{u}$ such that for $\sigma := C_{\exp H_\Theta}$, we have
\begin{equation}
	U_\Theta = \{ u\in U : C_{\exp H_\Theta} (u)=u \}.
	\end{equation}
	
	Moreover, let $G$ be the complex Lie group and $U$ be its compact real form. Then, the cotangent bundle of $\mathbb{F}_\Theta$ has a realization as a symmetric space.
If we use also the notation $\sigma$ for  $\mathrm{Ad}_\mathfrak{u}(\exp H_\Theta)$ and $\mathrm{Ad}_\mathfrak{g}(\exp H_\Theta)$, then $(\mathfrak{u} , \mathfrak{u}_\Theta ,\sigma)$ and  $(\mathfrak{g}, \mathfrak{z}_\mathfrak{g} (H_\Theta) , \sigma )$ are symmetric Lie algebras.
\end{ithm}
A detailed explicit description of $\bb F_{\Theta}$ as $U/U_{\Theta}$ is presented in Section \ref{sec:proofmain}. Related to that, we obtain the following result:
\begin{ithm}\label{thm:realizations}
    The following flag manifolds $\bb F_{\Theta}$ are symmetric spaces with corresponding symmetric cotangent bundle $T^*\bb F_{\Theta}$:
    \begin{center}
    \begin{tabular}{|c|c|}
    \hline
    $\bb F_{\Theta}$ &  $T^*\bb F_{\Theta}$\\
     \hline
      $\mathrm{SU}(p+q)/\mathrm{S}(\mathrm{U}(p)\times \mathrm{U}(q))$   &  $\mathrm{Sl}(p+q, \bb C)/\mathrm{S}(\mathrm{Gl}(p,\bb C)\times \mathrm{Gl}(q,\bb C))$\\
       \hline
       $\mathrm{Sp}(l)/\mathrm{U}(l)$ &  $\mathrm{Sp}(l,\bb C)/\mathrm{Gl}(l,\bb C)$\\
        \hline
      $\mathrm{SO}(2l)/\mathrm{U}(l)$ & $\mathrm{SO}_{\bb C}(2l)/\mathrm{Gl}(l,\bb C)$.\\
      \hline
    \end{tabular}
    \end{center}
\end{ithm}

Let us recall the definitions provided in \cite{BERGER}. In this context, an affine symmetric space denoted as $G/H$ (or a symmetric Lie algebra) is referred to as $\mathbb{C}$-symmetric if there exists a complex structure on the vector space $\mathfrak{m}$ that remains invariant under the linear representation ($\mathrm{Ad}(H), \mathfrak{m}$) [or (ad($\mathfrak{h}$), $\mathfrak{m}$)].

Furthermore, an affine symmetric space $G/H$ (or symmetric Lie algebra) is termed semi-K\"ahler if it satisfies the following criteria:
\begin{enumerate}
    \item It is $\mathbb{C}$-symmetric;
    \item For a specific complex structure on $\mathfrak{m}$, the representation ($\mathrm{Ad}(H), \mathfrak{m}$) [or (ad($\mathfrak{h}$), $\mathfrak{m}$)] preserves a non-degenerate Hermitian form defined on $\mathfrak{m}$.
\end{enumerate}
\begin{proposition}	\label{RRR} \cite{BERGER}
	Let $(\mathfrak{g},\mathfrak{h},\sigma)$ be a symmetric Lie algebra with $\mathfrak{g}$ simple. The symmetric Lie algebra is semi-K\"ahler if, and only if, the isotropy sub-algebra $\mathfrak{h}$ contains the Lie algebra $\mathbb{R}$.
\end{proposition}

If $(\mathfrak{g},\mathfrak{h},\sigma)$ is a symmetric Lie algebra, we name $(\mathfrak{g},\mathfrak{h})$ as a symmetric Lie pair. Berger gives a table of all symmetric Lie pairs indicating which are semi-K\"ahler. Some symmetric semi-K\"ahler Lie pairs are: $(\mathfrak{sl}(p+q,\mathbb{C}),\mathfrak{sl}(p,\mathbb{C}) \oplus \mathfrak{sl}(q,\mathbb{C}) \oplus \mathbb{C}^*)$, $(\mathfrak{sp}(n,\mathbb{C}),\mathfrak{sl}(n,\mathbb{C}) \oplus \mathbb{C}^*)$ and $(\mathfrak{so}(2n,\mathbb{C}),\mathfrak{sl}(n,\mathbb{C}) \oplus \mathbb{C}^*)$. Note that the isotropy sub-algebras contain the complex vector space $\mathbb{C}$ in those examples, hence containing $\mathbb{R}$. Theorem \ref{thm:mainintro} verifies that each of the semi-K\"ahler symmetric spaces above described are cotangent spaces of some flag manifolds endowed with a symmetric space structure.

In the context of a flag symmetric space $\mathbb{F}_\Theta$, its symmetric Lie algebra is denoted as $(\mathfrak{u},\mathfrak{u}_\Theta,\sigma)$. Here, we also use the notation $\sigma$ to represent the automorphism $\mathrm{Ad}(\exp H_\Theta)$ acting on $\mathfrak{g}$. This symmetric Lie algebra follows a canonical decomposition as $\mathfrak{u}=\mathfrak{u}_\Theta + \mathfrak{m}$.

Let us introduce the subset $\mathcal{S}$, defined as $\mathrm{Ad}(\exp(\sqrt{-1}\mathfrak{m}) H_\Theta)$ within the adjoint orbit $\mathrm{Ad}(G)H_\Theta$. One of our key findings asserts that $\mathcal{S}$ forms a submanifold within the adjoint orbit $\mathrm{Ad}(G)H_\Theta$ and it can be understood as the dual symmetric space of $\mathbb{F}_\Theta$. Furthermore, $\mathcal{S}$ constitutes a Riemannian manifold that is diffeomorphic to a vector space. Additionally, if a fiber within the cotangent bundle $T^*\mathbb{F}_\Theta$ intersects with $\mathcal{S}$, such a fiber and $\mathcal{S}$ intersect transversely.

In Sections \ref{sec:oexamplinhovemaqui} and \ref{DSS}, we present the Examples \ref{exm} and \ref{exemplorealparte2}. There we discuss the real flag manifold $\mathrm{SO}(2)/\{ \pm I \} \cong \mathrm{S}^1$, which is a symmetric space, such that its cotangent bundle has a realization as the one-sheeted hyperboloid $Q= \{ (x,y,z) \in \mathbb{R}^3: x^2 + y^2 -z^2 =1 \}$, that is also a symmetric space. The flag manifold $\mathrm{S}^1$ is the intersection of $Q$ with the plane $z=0$. We study the orbit of the exponential of symmetric zero-trace matrices in the matrix\[\left( \begin{array}{cc}
    1&0\\0&-1
\end{array} \right) .\] Such an orbit has a realization as the curve $\mathcal{C} = \{ (0,y,z)\in \mathbb{R}^3 : y^2 -z^2 =1, \; y>0 \}$, which is intuitively symmetric with respect to the axis $Y$. In the complex case, the generalization of this orbit is the submanifold $\mathcal{S}$. Hence, these examples work as ``toy models'' to the general description given by Theorem \ref{thm:mainintro}.

Remarkable as well, in general, the intersection of the submanifold $\mathcal{S}$ with the fibers of the bundle $\mathrm{Ad}(G)H_\Theta \rightarrow \mathbb{F}_\Theta$ is either null or a single point. Then, the projection of the cotangent bundle restricted to the submanifold $\mathcal{S}$ is an injective map since each fiber intersecting $\mathcal{S}$ does so in a single element. Then, $\mathcal{S}$ is continuously deformed into a part of the sphere $\mathbb{F}_\Theta$:

\begin{ithm}\label{thm:aS}
	Let $\mathbb{F}_\Theta = U/U_\Theta$ be a flag symmetric space with symmetric Lie algebra
	\begin{equation*}
		\mathfrak{u} = \mathfrak{u}_\Theta + \mathfrak{m}.
	\end{equation*}
	Then,
	\begin{enumerate}[$(a)$]
		\item  $\mathcal{S}$ is a submanifold of the adjoint orbit $G\cdot H_\Theta$, with $G$ being the complex Lie group with compact real form $U$. Furthermore, 
		\begin{equation*}
			\mathcal{S} \cong U^* /U_\Theta,
		\end{equation*}
		where $U^* / U_\Theta$ is the dual symmetric space of $\mathbb{F}_\Theta$
		\item the submanifold $\mathcal{S}$ is diffeomorphic to a vector space
		\item  $\cal S$ is a Riemannian homogeneous space with a Riemannian structure defined by the Cartan--Killing form
		\item  any fiber in the cotangent bundle of $\mathbb{F}_\Theta$ intersecting the submanifold $\mathcal{S}$ intersects it \emph{transversally}.
	\end{enumerate}        
\end{ithm}

Once this paper also deals with cotangent bundles, it is natural to define a symplectic form $$\omega (X, Y) = \mathrm{Im} \mathcal{K}(X,\sigma Y), \qquad X, Y\in \mathfrak{g},$$ where $\mathcal{K} (\cdot, \cdot )$ is the Cartan--Killing form and $\sigma$ is the conjugation in $\mathfrak{g}$ concerning the compact real form $\mathfrak{u}$. Since the adjoint orbit $\mathrm{Ad}(G)H_\Theta$ is contained in the Lie algebra $\mathfrak{g}$, the symplectic form  $\omega$ induces a symplectic form in the adjoint orbit through the pullback of $\omega$. The submanifold $\mathcal{S}$ and the flag symmetric space $\mathbb{F}_\Theta$ are Lagrangian submanifolds of $\mathrm{Ad}(G)H_\Theta$ with respect to $\omega$. However, the submanifold $\mathcal{S}$ is not a Lagrangian submanifold for the KKS form (Kostant-Kirillov-Souriau):

\begin{ithm}\label{thm:secondmainintro}
    	The symmetric flag space $\mathbb{F}_\Theta = U/U_\Theta$ and its dual symmetric space $\mathcal{S}$ are Lagrangian submanifolds of the adjoint orbit $G\cdot H_\Theta$ with the symplectic form $\omega$. However, $\cal S$ is not Lagrangian for the KKS form.
\end{ithm}

As a last main contribution, in Section \ref{sec:curvature}, we obtained the following application regarding the existence of some vector bundles over simply connected manifolds that do not admit metrics of nonnegative sectional curvature. However, which corresponding sphere bundles do admit, answering in negative Problem 9 in \cite{Ziller_fatnessrevisited} for connection metrics:

\begin{ithm}\label{thm:ziller}
For any symmetric flag space $U/U_{\Theta}$, there is a vector bundle $U^*/U_{\Theta}\rightarrow \widetilde M \rightarrow U/U_{\Theta}$ that does not admit a connection metric of nonnegative sectional curvature whose base is considered with the $-\cal K$ metric, that is, the $\mathrm{Ad}(U_{\Theta})$ metric induced by the negative of the Cartan Killing--form on $\lie u$. Moreover, any sphere bundle obtained considering spheres on the fiber  $U^*/U_{\Theta}$ has nonnegative sectional curvature.
\end{ithm}

The exciting aspect of Theorem \ref{thm:ziller} is due a kind of converse to Cheeger--Gromoll's Soul Theorem: \emph{any complete
non-compact manifold $M$ with a metric of nonnegative sectional curvature is diffeomorphic to the normal bundle of some compact totally geodesic manifold} $\Sigma \subset M$, named as \emph{Soul} of $M$. That is, \emph{which vector
bundles over compact, nonnegatively curved base spaces can admit complete metrics
of nonnegative curvature?} Until now, it was not known an example of a vector bundle over a simply connected Riemannian manifold whose associated sphere bundle admits a metric of nonnegative sectional curvature, although the corresponding vector bundle does not. The description here of $\cal S$ as the dual of a symmetric space of compact type allows us to construct the first example. However, as in \cite{GONZALEZALVARO201753}, or \cite{10.4310/jdg/1632506394}, up some connecting sums, these vector bundles admit metrics of nonnegative sectional curvature.

\section{Lie theoretical description of flag manifolds and symmetric spaces}

\subsection{The basics of root systems and parabolic Lie algebras}
\label{sec:roots}

Let us briefly review some fundamental facts about Lie algebras and establish the notation that will be consistently used throughout this paper. Let $\mathfrak{g}$ be a complex semisimple Lie algebra and $\mathfrak{h}$ a Cartan sub-algebra with \emph{roots set} denoted by $\Pi$. We recall that since $\mathbb C$ is algebraically closed and of zero characteristic, using that $\lie g$ is finite-dimensional, it follows that $\lie h$ is Abelian. Moreover, the adjoint representation $\mathrm{ad}: \lie g \rightarrow \lie{gl}(\lie g)$ enjoys the property that the image $\mathrm{ad}(\lie h)$ consists of \emph{semisimple operators}. Hence, all such operators are simultaneously diagonalizable. Considering this, we thus decompose $\lie g$ via appropriate invariant subspaces
\begin{equation}
    \lie g = \bigoplus_{\lambda \in \lie h^*}\lie g_{\lambda},
\end{equation}
where
\begin{equation}
    \lie g_{\lambda} := \ker \left(\mathrm{ad}(h) - \lambda(h) 1 : \lie g \rightarrow \lie g\right), 
\end{equation}
for $h \in \lie h$. Hence, the roots set is nothing but $\Pi := \left\{\alpha \in \lie h^{*}\setminus \{0\} : \lie g_{\alpha} \not\equiv 0  \right\}$. We reinforce that it can also be proved that for each $\alpha \in \Pi$, it holds that $\mathrm{dim}_{\mathbb C}\lie g_{\alpha} = 1$.

We can extract from $\Pi$ a subset $\Pi^{+}$ completely characterized by both:
\begin{enumerate}[(i)]
    \item for each root $\alpha\in \Pi$ only one of $\pm \alpha$ belongs to $\Pi^{+}$
    \item for each $\alpha, \beta \in \Pi^{+}$ necessarily $\alpha + \beta \in \Pi^{+}$ if $\alpha + \beta \in \Pi$.
\end{enumerate}
To the set $\Pi^{+}$, we name \emph{subset of positive roots}. We say that the subset $\Sigma \subset \Pi^{+}$ consists of the \emph{simple roots system} if it collects the positive roots, which can not be written as a combination of two elements in $\Pi^{+}$.

Since $\lie h$ is an abelian Lie sub-algebra, we can pick a basis $\{ H_\alpha : \alpha \in \Sigma\}$ to $\lie h$ and complete it with $\{ X_\alpha \in \mathfrak{g}_\alpha: \alpha \in \Pi \}$, generating what is called a \textit{Weyl basis} of $\mathfrak{g}$: for any $\alpha, \beta \in \Pi$
\begin{enumerate}
    \item $\cal K(X_{\alpha},X_{-\alpha}) := \mathrm{tr}~(\mathrm{ad}(X_{\alpha}\circ \mathrm{ad}(X_{-\alpha})) = 1$
    \item $[X_{\alpha},X_{\beta}] = m_{\alpha,\beta}X_{\alpha+\beta},~m_{\alpha,\beta}\in \mathbb{R}$,
\end{enumerate}
where $\cal K$ is the \emph{Cartan--Killing} form of $\lie g$, see \cite[p. 214]{San_Martin_2021}.

Then $\mathfrak{g}$ decomposes into root subspaces:
\begin{equation*}
	\mathfrak{g}=\mathfrak{h} + \sum_{\alpha \in \Pi} \mathbb{C} X_\alpha.
\end{equation*}
It is straightforward from \cite[Theorem 11.13, p. 224]{San_Martin_2021} that such a decomposition implies that the compact real form of $\mathfrak{g}$ is given by
\begin{equation}\label{eq:compactrealform}
	\mathfrak{u}=\sqrt{-1}\mathfrak{h}_{\mathbb{R}} + \sum_{\alpha \in \Pi^+} \mathrm{Span}_\mathbb{R} (X_\alpha + X_{-\alpha}, \sqrt{-1}(X_\alpha - X_{-\alpha}))
\end{equation}

Now recall that for any subset $Q\subset \Pi$, we have the decomposition $Q=Q^s \cup Q^a$, where $Q^s := Q \cap (-Q)$, and $Q^a := Q\setminus Q^s$. In what follows we will always denote by $\mathfrak{g}(Q)$ the algebra generated by $Q$
\begin{equation}\label{eq:generating}
	\mathfrak{g}(Q):=  \langle [X_\alpha, X_{-\alpha} ]=H_\alpha , \alpha \in Q^s \rangle + \sum_{\alpha \in Q} \mathbb{C} X_\alpha.
\end{equation} 

Once we have recalled several aspects of the roots system, let us explain how a system of roots, via a characteristic element, can provide a helpful description of parabolic sub-algebras.

Fix a subset $\Theta = \{ \alpha _1 , \cdots , \alpha_m \} \subset \Sigma$ of simple roots. We shall denote by $\langle \Theta\rangle$ the set of roots generated by $\Theta$. Also, we denote the intersection $\langle \Theta \rangle \cap \Pi ^{\pm}$ by $\langle \Theta \rangle ^{\pm}$. For each subset $\Theta $ we get the decomposition of $\mathfrak{g}$:
\begin{equation}	\label{ggd}
	\mathfrak{g} = \mathfrak{n}^-_\Theta \oplus \mathfrak{z}_\Theta \oplus \mathfrak{n}_\Theta^+,
\end{equation}
where
\begin{equation}\label{ntmes}
	\mathfrak{n}^-_\Theta = \mathfrak{g} (\Pi^- \setminus \langle \Theta \rangle ^-), \,\,\,\, \mathfrak{n}^+_\Theta = \mathfrak{g} (\Pi^+ \setminus \langle \Theta \rangle ^+)
\end{equation}
are nilpotent sub-algebras of $\mathfrak{g}$, and $\mathfrak{z}_\Theta = \mathfrak{h}+\mathfrak{g}(\Theta)$
is a reductive sub-algebra, with 
\begin{equation*}
	\mathfrak{g} (\Theta) := \mathfrak{g} ( \langle \Theta \rangle ) = \langle H_\alpha \rangle _{ \alpha \in \langle \Theta \rangle } + \sum_{\alpha \in \langle \Theta \rangle} \mathbb{C} X_\alpha.
\end{equation*}
We also have $\mathfrak{z}_\Theta=\mathfrak{z}_{\mathfrak{g}} (H_\Theta)$.

 Now one can construct a parabolic sub-algebra from a subset $\Theta \subset \Sigma$ of simple roots of $\mathfrak{g}$: 
\begin{equation}
	\mathfrak{p}_\Theta= \mathfrak{h} + \mathfrak{g}(\langle \Theta \rangle ^- \cup \Pi ^+ ) = \mathfrak{z}_\Theta + \mathfrak{n}_\Theta^+ = \sum_{\alpha \in  \langle \Theta \rangle ^-} \mathfrak{g}_\alpha + \mathfrak{b},
\end{equation}
where $\lie b$ is a Borel sub-algebra.

We also recall that a parabolic subgroup $P_\Theta$, with Lie algebra ${\lie p}_{\theta}$, is the normalizer of $\mathfrak{p}_\Theta$ in $G$, i.e., $P_\Theta =\{ g\in G : \mathrm{Ad} (g) \mathfrak{p}_\Theta = \mathfrak{p}_\Theta \}$

\begin{definition}\label{def:flag}
	A \emph{flag manifold} of a complex semisimple Lie group $G$ is the homogeneous space $M=G/P_\Theta$, where $P_\Theta$ is a parabolic subgroup of $G$.
\end{definition}

If $U$ corresponds to the compact real form of $G$ then it can be shown that $\bb F_{\Theta} = U/U_{\Theta}$ where $U_{\Theta} = U\cap P_{\Theta}$. 

Given $\Theta \subset \Sigma$, an element $H_{\Theta} \in \mathfrak{h}$ is said to be \textit{characteristic} for $\Theta$ if one recognizes $\Theta = \{ \alpha \in \Sigma : \alpha (H_\Theta) =0 \}$. In particular, such a subset $\Theta $ defines a parabolic sub-algebra $\mathfrak{p}_\Theta$, that is, it contains a maximal solvable sub-algebra of $\lie g$, i.e., a Borel sub-algebra, with parabolic subgroup $P_\Theta$. It also defines a flag manifold $\mathbb{F}_\Theta = G/P_\Theta$. It can be checked that the parabolic sub-algebra is $\mathfrak{p}_\Theta = \oplus _{\lambda \geq 0} \mathfrak{g}_\lambda$, where $\lambda$ runs through the non-negative eigenvalues of $\mathrm{Ad}(H_\Theta)$. Conversely, starting with $H_0 \in \mathfrak{h}$ we define $\Theta _{H_0} = \{ \alpha \in \Sigma : \alpha (H_0)=0 \}$, leading to the corresponding flag manifold $\mathbb{F}_{H_0} := \mathbb{F}_{\Theta_{H_0}}$.

\subsection{Flag Symmetric Spaces and their algebraic aspects}

\label{sec:simalgebras}
This section will explore the Lie algebraic description of symmetric spaces, such as given in \cite[Chapter XI]{KN}, jointly with their algebraic descriptions of parabolic sub-algebras, to study symmetric spaces that are also flag manifolds. More precisely, we call \emph{symmetric flag spaces} to symmetric spaces that are also flag manifolds. 

We recall that a \textit{symmetric space} is a triple $(G, H,\sigma)$ consisting of a connected Lie group $G$, a closed subgroup $H$ of $G$, and an involutive automorphism $\sigma$ of $G$ such that $H$ lies between $G_\sigma$ and the identity component of $G_\sigma$, where $G_\sigma$ denotes the closed subgroup of $G$ consisting of all elements left fixed by $\sigma$. An infinitesimal version of a symmetric space, named as a \textit{symmetric Lie algebra}, is a triple $(\mathfrak{g}, \mathfrak{h}, \sigma )$ consisting of a Lie algebra $\mathfrak{g}$, a sub-algebra $\mathfrak{h}$ of $\mathfrak{g}$, and an involutive automorphism $\sigma$ of $\mathfrak{g}$ such that $\mathfrak{h}$ consists of all elements of $\mathfrak{g}$ which are left fixed by $\sigma$. 

Every symmetric space $(G,H,\sigma)$ gives rise to a symmetric Lie algebra $(\mathfrak{g}, \mathfrak{h}, \sigma)$ in a natural manner; $\mathfrak{g}$ and $\mathfrak{h}$ are the Lie algebras of $G$ and $H$, respectively, and the automorphism $\sigma$ of $\mathfrak{g}$ is the one induced by the automorphism $\sigma$ of $G$. Conversely, if $(\mathfrak{g}, \mathfrak{h}, \sigma )$ is a symmetric Lie algebra and $G$ is a simply connected Lie group with Lie algebra $\mathfrak{g}$, the automorphism $\sigma$ of $\mathfrak{g}$ induces an automorphism $\sigma$ of $G$. Moreover, for any subgroup $H$ lying between $G_\sigma$ and the identity component of $G_\sigma$, the triple is a symmetric space, and $(G, H)$ is a symmetric pair.

Let $(\mathfrak{g}, \mathfrak{h}, \sigma )$ be a symmetric Lie algebra. Since $\sigma$ is involutive, seeing it as a linear operator, one gets that its eigenvalues are $1$ and $-1$. One then recovers $\mathfrak{h}$ as the eigenspace associated with $1$. Let $\mathfrak{m}$ be the eigenspace associated with $-1$. The decomposition 
\begin{equation}
	\mathfrak{g} =  \mathfrak{h} + \mathfrak{m}
\end{equation} is called the \emph{canonical decomposition} of $( \mathfrak{g} , \mathfrak{h} ,\sigma) $.

Since the following shall play some role, it is worth recalling that
\begin{proposition}		\label{colchetes}
	\cite[Proposition 2.1, p. 226]{KN} If $\mathfrak{g} = \mathfrak{h} + \mathfrak{m}$ is the canonical decomposition of a symmetric Lie algebra $(\mathfrak{g}, \mathfrak{h}, \sigma)$, then 
	\begin{equation}
		[\mathfrak{h} , \mathfrak{h}] \subset \mathfrak{h} , \quad [ \mathfrak{h} , \mathfrak{m} ] \subset \mathfrak{m}, \quad [ \mathfrak{m} , \mathfrak{m} ] \subset \mathfrak{h}.
	\end{equation}
\end{proposition}

\begin{proposition}	\label {dgrosym}\cite[Proposition 2.2, p. 227]{KN}
	Let $(G,H, \sigma)$ be a symmetric space and $(\mathfrak{g},\mathfrak{h},\sigma)$ its symmetric Lie algebra. If $\mathfrak{g} = \mathfrak{h} + \mathfrak{m}$ is the canonical decomposition of $(\mathfrak{g} , \mathfrak{h},\sigma)$, then \begin{equation}
		\mathrm{Ad}(H)\mathfrak{m} \subset \mathfrak{m}.
	\end{equation}
\end{proposition}

Let $\bb F_{\Theta}$ be a flag manifold, as defined in Definition \ref{def:flag}. It is remarkable, considering the discussions on Sections \ref{sec:roots} and \ref{sec:simalgebras}, that for flag-symmetric spaces, there exists an involutive inner automorphism $\sigma$ over $\mathfrak{g}$ (unless explicitly said, we use the same notation for the automorphism over $G$, $\mathfrak{u}$ and $U$) such that the $\sigma$-fixed point set in $\mathfrak{g}$ is
\begin{equation}
	\mathfrak{g}^\sigma = \mathfrak{z}_\Theta = \mathfrak{z}_\mathfrak{g} (H_\Theta),
\end{equation}
where we do recall that $\lie z_{\lie g}(H_{\theta})$ is obtained by the description given by equation \eqref{eq:generating}. Then, the tangent space of the homogeneous space $G/G^\sigma$ at the origin can be identified with
\begin{equation*}
	\mathfrak{m}_G \cong \mathfrak{n}_\Theta^- \oplus \mathfrak{n}_\Theta^+,
\end{equation*}
(if needed, recall the definitions of both $\mathfrak{n}_\Theta^{\pm}$ on equation \eqref{ntmes}).

Hence, in this context, the canonical decomposition (concerning Proposition \ref{colchetes}) of the complex Lie algebra $\lie g$ is given by
\begin{equation}\label{eq:decompoeaqui}
	\mathfrak{g}= \mathfrak{z}_\Theta \oplus \mathfrak{m}_G.
\end{equation}

\section{The Cotangent bundle of a symmetric flag space}

As mentioned in the Introduction, in \cite{ADORB}, the adjoint orbit $G\cdot H_\Theta :=\mathrm{Ad}(G)H_\Theta$ realizes the cotangent bundle of the complex flag manifold $\mathbb{F}_\Theta$.
\begin{theorem}	\label{orbitafibradoSM}
	\cite[Theorem 2.1]{ADORB} The adjoint orbit $\mathcal{O}(H_\Theta) =\mathrm{Ad}(G)\cdot H_\Theta \approx G/Z_\Theta$ of the characteristic element $H_\Theta$ is a $C^\infty$ vector bundle over $\mathbb{F}_\Theta$ that is isomorphic to the cotangent bundle $T^*\mathbb{F}_\Theta$. Moreover, we can write down a diffeomorphism $\iota :\mathrm{Ad}(G)\cdot H_\Theta \rightarrow T^*\mathbb{F}_\Theta$ such that
	\begin{enumerate}[$(a)$]
		\item $\iota$ is equivariant with respect to the actions of $U$, that is, for all $u\in U$,
		\begin{equation*}
			\iota \circ \mathrm{Ad} (u) = \tilde{u} \circ \iota,
		\end{equation*} 
		where $\tilde{u}$ is the lifting to $T^* \mathbb{F}_\Theta$ (via the differential) of the action of $u$ on $\mathbb{F}_\Theta$; and
		\item \label{item2lino} the pullback of the canonical simplectic form on $T^*\mathbb{F}_\Theta$ by $\iota$ is the (real) Kostant-Kirillov-Souriau form on the orbit.
	\end{enumerate}
\end{theorem}

On the one hand, the realization of the cotangent bundle is generally by the coadjoint orbit. However, we claim that such an orbit can be identified with the adjoint orbit, also realized by the homogeneous space $G/Z_\Theta$, where $Z_\Theta$ is the centralizer of $H_\Theta$ in $G$ by the adjoint representation. The precise reason for this \emph{exchangeability} on the possible representations is because \emph{for semisimple Lie algebras, the adjoint representation behaves like the coadjoint representation}.

Indeed, we recall that the \emph{coadjoint representation} $\mathrm{Ad}^*$ of $G$ on the dual $\mathfrak{g}^*$ of the Lie algebra $\mathfrak{g}$ is defined by $\mathrm{Ad}^*(g)\alpha =\alpha \circ \mathrm{Ad}(g^{-1}),~ g\in G$ and $\alpha\in \mathfrak{g}^*$. Moreover, its infinitesimal representation $\mathrm{ad}^* : \mathfrak{g} \rightarrow \mathfrak{gl}(\mathfrak{g})$ is given by $\mathrm{ad}^*(X)\alpha =-\alpha \circ \mathrm{Ad}(X)$. Hence, $G$ acts on $\mathfrak{g}^*$ by the representation $\mathrm{Ad}^*$, being the induced action vector fields $\tilde{X}$, $X\in \mathfrak{g}$, given by $\tilde{X}= \mathrm{ad}^* (X)$.

For $\alpha \in \mathfrak{g}^*$, the coadjoint orbit $\mathrm{Ad}^*(G)\alpha$ is identified with the homogeneous space $G/Z_\alpha$, where $Z_\alpha$ is the closed subgroup 
\begin{equation*}
	Z_\alpha = \{ g\in G :\alpha \circ \mathrm{Ad}(g^{-1}) = \alpha \}. 
\end{equation*}
The Lie algebra $\mathfrak{z}_\alpha$ of $Z_\alpha$ is given by $\mathfrak{z}_\alpha = \{ X\in \mathfrak{g} : \alpha \circ \mathrm{ad}(X)=0 \}$. Moreover, the tangent space of $\mathrm{Ad}^*(G) \alpha$ at $\alpha$ is given by
\begin{equation*}
	T_\alpha ( \mathrm{Ad} ^*(G) \alpha ) = \{ \mathrm{ad}^*(X)\alpha : X\in \mathfrak{g} \}.
\end{equation*}

	Following \cite[p. 324]{San_Martin_2021}, the Cartan--Killing form $\mathcal{K}(\cdot , \cdot)$ of a Lie algebra $\mathfrak{g}$ of the semisimple group $G$ is non-degenerate and define an isomorphism $W: \mathfrak{g} \rightarrow \mathfrak{g}^*$ by $W(X)(\cdot) = \mathcal{K} (X,\cdot)$. This isomorphism exchanges the adjoint and coadjoint representations, i.e., \[W\mathrm{Ad}(g)=\mathrm{Ad}^*(g)W\] for every $g\in G$, since $\mathcal{K}$ is $\mathrm{Ad}$-invariant. Hence, $W$ applies diffeomorphically the adjoint orbits in the coadjoint orbits.

Concerning Theorem  \ref{orbitafibradoSM}, the projection $\pi : G\cdot H_\Theta \rightarrow \mathbb{F}_\Theta$ is obtained via the action of $G$: the canonical fibration $gZ_\Theta \in G/Z_\Theta \mapsto gP_\Theta \in \mathbb{F}_\Theta$ has as fiber $P_\Theta /Z_\Theta$. In this manner, in terms of the adjoint representation, the fiber is $\mathrm{Ad}(P_\Theta)\cdot H_\Theta$, which coincides with the affine subspace $H_\Theta + \mathfrak{n}_\Theta^+$ (see \cite[p. 364]{ADORB}).

Moreover, in \cite[Section 2.2, p. 365]{ADORB} it is also described the isomorphism of the adjoint orbit $G\cdot H_\Theta$ with the cotangent bundle $T^*\mathbb{F}_\Theta$. The tangent space $T_{b_\Theta}\mathbb{F}_\Theta$ of the flag manifold $\mathbb{F}_\Theta$ at the origin $b_\Theta$ can be identified with $\mathfrak{n}_\Theta^-$ and the isotropy representation $U_\Theta \rightarrow $Gl$(T_{b_\Theta}\mathbb{F}_\Theta)$ becomes the restriction of the adjoint representation. Hence, the subspace $\mathfrak{n}^+_\Theta$ is isomorphic to the dual $(\mathfrak{n}_\Theta^-)^*$ of $\mathfrak{n}_\Theta^-$ via the Cartan--Killing form $\mathcal{K}(\cdot , \cdot)$ of $\mathfrak{g}$, being the map $X\in \mathfrak{n}_\Theta^+ \mapsto \mathcal{K}(X,\cdot ) \in (\mathfrak{n}_\Theta^-)^*$ an isomorphism.

\subsection{Characterization of some symmetric spaces and flag manifolds}

We begin recalling that, according to \cite[Chapter XI.7]{KN}, the \emph{orthogonal symmetric Lie algebras} are of compact and non-compact type. Moreover, they are of four classes (these former concepts shall be approached in further detail in Section \ref{sec:curvature}). The types of symmetric spaces are related to their geometric properties as follows
\begin{theorem} 
	\cite[Theorem 8.6, p. 256]{KN} Let $(G,H,\sigma)$ be a symmetric space with $\mathrm{ad}_{\lie g}(\lie h)$ compact and let $(\mathfrak{g},\mathfrak{h},\sigma)$ be its orthogonal symmetric Lie algebra. Take any $G$-invariant Riemannian metric on $G/H$. Then, 
	\begin{itemize}
		\item[(1)] if $(\mathfrak{g},\mathfrak{h},\sigma)$ is of compact type, then $G/H$ is a compact Riemannian symmetric space with nonnegative sectional curvature and positive definite Ricci tensor
		\item[(2)] if $(\mathfrak{g},\mathfrak{h},\sigma)$ is of non-compact type, then $G/H$ is a simply connected non-compact Riemannian symmetric space with non-positive sectional curvature and negative definite Ricci tensor and is diffeomorphic to a Euclidean space.
	\end{itemize}
\end{theorem}

On the one hand, on \cite[Table V, p.518]{HELGA} we can find the following table of irreducible Riemannian globally symmetric classical spaces of type I and III, with $p+q=n$:

\begin{center} \label{Table1}
	\begin{tabular}{ | m{2em} | m{5.3cm}| m{5.3cm} |} 
		\hline
		& Compact symmetric space & Noncompact symmetric space \\  	
		\hline
		AI & $\mathrm{SU}(n)/\mathrm{SO}(n)$ & $\mathrm{Sl}(n,\mathbb{R})/\mathrm{SO}(n)$ \\ 
		\hline
		AII & $\mathrm{SU}(2n)/\mathrm{Sp}(n)$ & $\mathrm{SU}^* (2n)/\mathrm{Sp}(n)$  \\ 
		\hline
		AIII & $\mathrm{SU}(p+q)/\mathrm{S}(\mathrm{U}(p)\times \mathrm{U}(q))$ & $\mathrm{SU}(p,q)/\mathrm{S}(\mathrm{U}(p) \times \mathrm{U}(q))$ \\ 
		\hline
		BDI & $\mathrm{SO}(p+q)/\mathrm{SO}(p)\times \mathrm{SO}(q)$ & $\mathrm{SO}_o (p,q)/\mathrm{SO}(p)\times \mathrm{SO}(q)$ \\ 
		\hline
		DIII & $\mathrm{SO}(2n)/\mathrm{U}(n)$ & $\mathrm{SO}^*(2n)/\mathrm{U}(n)$ \\ 
		\hline
		CI & $\mathrm{Sp}(n)/\mathrm{U}(n)$ & $\mathrm{Sp}(n,\mathbb{R})/\mathrm{U}(n)$ \\ 
		\hline
		CII & $\mathrm{Sp}(p+q)/\mathrm{Sp}(p)\times \mathrm{Sp}(q)$ & $\mathrm{Sp}(p,q)/\mathrm{Sp}(p)\times \mathrm{Sp}(q)$ \\ 
		\hline
	\end{tabular}
\end{center}

On the other hand, an essential class of symmetric spaces generalizing the Riemannian symmetric spaces is the one of \emph{pseudo-Riemannian symmetric spaces}, in which an indefinite Riemannian metric replaces the Riemannian metric, leading, for instance, to Lorentzian symmetric spaces. Generally, symmetric and locally symmetric spaces can be seen as affine symmetric spaces. As shall be further approached in Section \ref{sec:curvature}, if $G/H$ is a symmetric space, we can provide some $G$-invariant natural connection associated with $G/H$. Conversely, a manifold with such a connection is locally symmetric, and it is a symmetric space up to passing to its universal covering.

Such manifolds can also be described as those affine manifolds whose \emph{geodesic symmetries} are all globally defined affine diffeomorphisms, generalizing the Riemannian and pseudo-Riemannian case. In \cite[Tableau I, p. 114]{BERGER} we find the following table of symmetric spaces for each \emph{classical} complex simple Lie group:
\begin{center} \label{Table2}
	\begin{tabular}{ | m{4em} | m{4.4cm}| m{4.8cm} | m{1.8cm} | } 
		\hline
		G &  & & \\ 	
		\hline
				$\mathrm{Sl}(n,\mathbb{C})$ & $G/\mathrm{SO}(n,\mathbb{C})$ & $G/\left(\mathrm{S}(\mathrm{Gl}(k,\mathbb{C})\times \mathrm{Gl}(l,\mathbb{C}))\right)$, $k+l=n$ & $G/\mathrm{Sp}(n,\mathbb{C})$ \\ 
		\hline
		$\mathrm{SO}(n,\mathbb{C})$ & $G/\left(\mathrm{SO}(k,\mathbb{C})\times \mathrm{SO}(l,\mathbb{C})\right)$, $k+l=n$ & $G/\mathrm{Gl}(n/2, \mathbb{C})$, n even &  \\ 
		\hline
		$\mathrm{Sp}(n,\mathbb{C})$ & $G/\left(\mathrm{Sp}(k,\mathbb{C})\times \mathrm{Sp}(l,\mathbb{C})\right)$, $k+l=n$ & $G/\mathrm{Gl}(n,\mathbb{C})$ & \\ 
		\hline
	\end{tabular}
\end{center} 

The next table shows the flag manifolds of the classical Lie groups. Here, consider $n=n_1+\cdots + n_s$, $n_1 \geq \cdots \geq n_s \geq 1$ and $l=l_1 + \cdots + l_k +m$, $l_1\geq \cdots \geq l_k \geq 1$, $k,m \geq 0$.

\begin{center}	\label{flags}
	\begin{tabular}{ | m{2em} | m{6.3cm}| } 
		\hline
		& Flag manifold  \\ 
		\hline
		
		A & $\mathrm{SU}(n)/\mathrm{S}(\mathrm{U}(n_1)\times \cdots \times \mathrm{U}(n_s)) $\\ 
		\hline
		B & $\mathrm{SO}(2l+1)/\mathrm{U}(l_1)\times \cdots \times \mathrm{U}(l_k) \times \mathrm{SO}(2m+1)$ \\ 
		\hline
		C & $\mathrm{Sp}(l)/\mathrm{U}(l_1)\times \cdots \times \mathrm{U}(l_k) \times \mathrm{Sp}(m)$  \\ 
		\hline
		D & $\mathrm{SO}(2l)/\mathrm{U}(l_1)\times \cdots \times \mathrm{U}(l_k) \times \mathrm{SO}(2m)$ \\ 
		\hline
	\end{tabular}
\end{center}

\

Comparing the tables above, we can see that
\begin{lemma}\label{lem:flagsym}
The symmetric spaces as in Theorem \ref{flags} that are also flag manifolds are 
\[\mathrm{SU}(p+q)/\mathrm{S}(\mathrm{U}(p)\times \mathrm{U}(q)),~~~ \mathrm{Sp}(l)/\mathrm{U}(l)~~~\text{and}~~~\mathrm{SO}(2l)/\mathrm{U}(l),\]of types \emph{A,C} and \emph{D}, respectively.
\end{lemma}

	\begin{remark} 
It is worth mentioning that we are frequently identifying	the space $\mathfrak{sl}(n,\mathbb{C})\oplus \mathbb{C}^*$ with $\mathfrak{gl}(n,\mathbb{C})$. Precisely, following \cite{BERGER}, we make such identification  as follows:
let $A\in \mathfrak{gl}(n,\mathbb{C})$. Consider the injective $\cal L$ map between $\mathfrak{gl}(n,\bb C)$ and $\mathfrak{sl}(n,\mathbb{C})\oplus \mathbb{C}^*$ given by: 

For $A\in \mathfrak{sl}(n,\mathbb{C})$ we make \[\cal L : A= \left[ A+\left( \begin{array}{cc}
	-1&0 \\ 0&\frac{1}{n-1} I_{n-1} \end{array} \right) \right]   + \left( \begin{array}{cc}
	1&0\\0&-\frac{1}{n-1}I_{n-1}
\end{array} \right) \mapsto A + \frac{1}{1-n} \in \mathfrak{sl}(n,\mathbb{C}) \oplus \mathbb{C}^*.
\] 

For $A\in \mathfrak{gl}(n,\bb C)\ominus\mathfrak{sl}(n,\mathbb{C})$ we let \[\cal L : A= \left[  A-\frac{\mathrm{tr}~A}{n}I_n \right]  + \frac{\mathrm{tr}~A}{n}I_n \mapsto \left[  A-\frac{\mathrm{tr}~A}{n}I_n \right] + \frac{\mathrm{tr}~A}{n} \in \mathfrak{sl}(n,\mathbb{C}) \oplus \mathbb{C}^*.\] The map $\cal L$ assures that $ \mathfrak{gl}(n,\mathbb{C}) \subset \mathfrak{sl}(n,\mathbb{C})\oplus \mathbb{C}^* $ and the other containment is straightforward.
\end{remark}

\subsection{Symmetric spaces on the adjoint orbit}
\label{sec:oexamplinhovemaqui}

 The flag manifolds given by Lemma \ref{lem:flagsym} are Riemannian symmetric manifolds. If we denote by $(\mathfrak{u},\mathfrak{k},\sigma)$ their orthogonal Lie algebra of compact type, one gets that these are of class (I). Next, we shall see these spaces as homogeneous spaces with canonical connections; see \cite[Theorem 3.1, p. 230]{KN}.

In Section \ref{sec:proofmain}, we shall prove Theorem \ref{cap345} below, i.e., that the cotangent bundle of a symmetric flag space is a symmetric space. 

\begin{theorem}	\label{cap345}
	Let $\mathbb{F}_\Theta$ be a symmetric flag space identified with the homogeneous space $\mathbb{F}_\Theta = U/U_\Theta$.
Then, there exists an element $H_\Theta$ in a fixed Cartan sub-algebra of the compact Lie algebra $\mathfrak{u}$ such that for $\sigma := C_{\exp H_\Theta}$, 
\begin{equation}
	U_\Theta = \{ u\in U : C_{\exp H_\Theta} (u)=u \}.
\end{equation}

Moreover, let $G$ be the complex Lie group such that $U$ is its compact real form. Then,  the cotangent bundle of $\mathbb{F}_\Theta$ has a realization as a symmetric space.

If we also use the notation $\sigma$ for  $\mathrm{Ad}_\mathfrak{u}(\exp H_\Theta)$ and $\mathrm{Ad}_\mathfrak{g}(\exp H_\Theta)$, then $(\mathfrak{u} , \mathfrak{u}_\Theta ,\sigma)$ and  $(\mathfrak{g}, \mathfrak{z}_\mathfrak{g} (H_\Theta) , \sigma )$ are symmetric Lie algebras. %\emph{COMPLETARRRRRRRRR}
\end{theorem}

According to Theorem \ref{cap345}, it is possible to find an element $H_\Theta$ in a Cartan sub-algebra of $\lie g$, determining the automorphism $\sigma$ for two symmetric spaces, and then, for two symmetric Lie algebras.

Since the Lie groups in this paper considered are all semisimple, the Cartan--Killing form $\mathcal{K}(\cdot, \cdot )$ is a non-degenerate $\mathrm{Ad}(H)$-invariant bilinear form, where $H$ is the isotropy subgroup for symmetric flag spaces and adjoint orbits. In the case of symmetric flag spaces, the Cartan Killing form is negative definite. Hence, $-\mathcal{K}(\cdot , \cdot )$ generates a $U$-invariant Riemannian metric on $\mathbb{F}_\Theta = U/U_\Theta$. However, for the cotangent bundle $T^*\mathbb{F}_\Theta \cong G/Z_\Theta$, there exists no $\mathrm{Ad}(Z_\Theta)$-invariant (definite) metric, because $Z_\Theta$ is never a group of isometries.

The proof of Theorem \ref{cap345} is nothing but a generalization of the Examples \ref{exm}--\ref{exemplorealparte2} below. We first present a real unidimensional flag to know $ \mathrm{SO}(2)/\{ \pm \mathrm{Id} \}$, which is a symmetric space such that its dual symmetric space is contained in the bi-dimensional adjoint orbit diffeomorphic to the cotangent bundle of $ \mathrm{SO}(2)/\{ \pm \mathrm{Id} \}$. 
\begin{example}	\label{exm}
	Let us consider the following basis of $\mathfrak{sl}(2,\mathbb{R})$:
	\begin{equation}	\label{basedesl}
		A=  \left( \begin{array}{cc}
			0&1\\1&0
		\end{array} \right) , \quad B=  \left( \begin{array}{cc}
			1&0\\0&-1
		\end{array} \right), \quad C=  \left( \begin{array}{cc}
			0&1\\-1&0
		\end{array} \right).
	\end{equation}

The compact Lie group $$\mathrm{SO}(2) = \left\{ \left( \begin{array}{cc}
	\cos s &  -\sin s \\ \sin s & \cos s
\end{array} \right) :s\in \mathbb{R} \right\} $$ determines the maximal flag manifold 
\begin{align}
	\mathrm{SO}(2) \cdot B =&	\left\{ \left( \begin{array}{cc}
		\cos s &  -\sin s \\ \sin s & \cos s
	\end{array} \right) \cdot B = \sin (2s) A + \cos (2s) B + 0 C : s\in \mathbb{R} \right\} \nonumber \\ =&  \mathrm{S^1} \label{S1}
\end{align} contained in the real vector space $\mathfrak{sl}(2,\mathbb{R}) \ominus \mathfrak{so}(2) =\mathrm{span}_{\bb R}\{ A,B \}$. This flag manifold is symmetric since it is the $1$-sphere $\mathrm{S^1}$.

 The adjoint orbit $\mathrm{Sl}(2,\mathbb{R}) \cdot B := \mathrm{Ad}(\mathrm{Sl}(2,\mathbb{R}))B$ consists of the set of matrices $xA+yB+zC$ for $x,y,z \in \mathbb{R}$, with eigenvalues $\pm 1$, i.e.,
 \begin{equation}
 	\mathrm{det} \left( \begin{array}{cc}
 		y&x+z\\x-z&-y
 	\end{array} \right) = \mathrm{det}  \left( \begin{array}{cc}
 	1&0\\0&-1
 \end{array} \right).
 \end{equation} Then, the adjoint orbit $\mathrm{Sl}(2,\mathbb{R})\cdot B$ is the one-sheeted hyperboloid \begin{equation} \label{hyperboloide}
 x^2 +y^2 -z^2 = 1.
\end{equation} 
This manifold has a realization as the homogeneous space \[\mathrm{Sl}(2,\mathbb{R})/\mathrm{S}(\mathrm{Gl}(1,\mathbb{R})\times \mathrm{Gl}(1,\mathbb{R} )) \newline \cong\mathrm{Sl}(2,\mathbb{R})/\mathbb{R}^*,\] where $\mathbb{R}^* = \mathbb{R}\setminus \{ 0 \}$.

 The cotangent bundle of $\mathrm{S^1}$ is also the union of fibers  $T^* \mathrm{S^1} = \bigcup_{k\in \mathrm{SO}(2)} k\cdot (B+ \mathfrak{n}^+ )$, with $\mathfrak{n}^+$ being the Lie sub-algebra of $\mathfrak{sl}(2,\mathbb{R})$ defined by the upper triangular matrices. For each element \begin{equation} \label{kexm}
 	k= \left( \begin{array}{cc}
 		\cos t & \sin t \\ -\sin t & \cos t
 	\end{array} \right) \in \mathrm{SO}(2), \quad t\in \mathbb{R},
 \end{equation} the fiber $k\cdot (B + \mathfrak{n}^+)$ is the one-dimensional affine vector space
 \begin{align}
 	\mathrm{Ad}(k) (B + \mathfrak{n}^+) :=& \left\lbrace   \left( \begin{array}{cc}
 		\cos t & \sin t \\ -\sin t & \cos t
 	\end{array} \right) \cdot \left[ B+ \left( \begin{array}{cc}
 	0&r\\0&0
 \end{array} \right)  \right] :r\in \mathbb{R} \right\rbrace  \nonumber \\
=& \left\lbrace  \left( \begin{array}{cc}
	\cos (2t)-\frac{r}{2} \sin (2t) & -\sin (2t) +r ( \cos t)^2 \\ -\sin (2t) -r ( \sin t )^2 & - \cos (2t) -\frac{r}{2} \sin (2t)
\end{array} \right)  : r\in \mathbb{R} \right\rbrace \nonumber \\
 =& \left\{ (	-\sin (2t)+ \frac{r}{2} \cos (2t)) A + (\cos (2t) + \frac{r}{2} \sin (2t) ) B +\frac{r}{2} C : r  \in \mathbb{R} \right\} . \label{fiberexmreal}
 \end{align} 
A vector equation defining each fiber above is
\begin{equation}  \left( \begin{array}{c}
		x\\y\\z
	\end{array} \right) = \left(
	 \begin{array}{c}
-	\sin(2t) \\ \cos (2t) \\ 0
\end{array}\right) + \frac{r}{2} \left( \begin{array}{c}
\cos(2t) \\ \sin (2t) \\ 1
\end{array}\right), \quad r\in \mathbb{R}. \label{elementodelafibra}
\end{equation}

\end{example}

\begin{example}	\label{Examplecomplexo}
Here, we present an example of a low-dimension symmetric space, to be further studied in Section \ref{sec:proofmain}. Consider the Lie group $G=\mathrm{Sl}(2,\mathbb{C})$, with compact real form $U=\mathrm{SU}(2)$ and take the element in the Lie algebra $\mathfrak{u}= \mathfrak{su}(2)$ 
\begin{equation}\label{eq:Htheta}
H_\Theta=\frac{ \sqrt{-1} \pi}{2} \left( \begin{array}{cc}
		1 & 0 \\ 0& -1
	\end{array} \right) .\end{equation}
	This element defines the maximal flag manifold $\mathbb{F}_\Theta=U/U_\Theta$ for $\Theta = \emptyset$, where $U_\Theta$ is the centralizer of $H_\Theta$ in $U$. The Lie group
\begin{equation}	\label{SU(2)}
	\mathrm{SU}(2)= \left\lbrace 
\left( \begin{array}{cc}
	\alpha & \beta \\- \bar{\beta} & \bar{\alpha}
\end{array} \right) :  \lvert \alpha \rvert ^2+ \lvert \beta \rvert ^2=1, \alpha , \beta \in \mathbb{C} 
\right\rbrace 
\end{equation} is diffeomorphic to the $3$-sphere $\mathrm{S^3}$ and $Z_U(H_\Theta)$ is the subgroup
\begin{eqnarray}
	\mathrm{S}(\mathrm{U}(1)\times \mathrm{U}(1)) =& \left\lbrace 
	\left( \begin{array}{cc}
		\alpha & 0 \\0 & \bar{\alpha}
	\end{array} \right) :  \lvert \alpha \rvert ^2=1, \alpha \in \mathbb{C} 
	\right\rbrace   \nonumber \\
	=& \left\lbrace 
	\left( \begin{array}{cc}
		e^{i\theta} & 0 \\ 0 & e^{-i\theta}
	\end{array} \right) : \theta \in [0,2\pi]
	\right\rbrace   \nonumber 
\end{eqnarray}
diffeomorphic to the unidimensional sphere $\mathrm{S^1}$.

Then, the flag manifold $\mathbb{F}_\Theta$ is a symmetric space and is the Riemann sphere
\begin{equation*}
	\mathrm{SU}(2)/\mathrm{S}(\mathrm{U}(1)\times \mathrm{U}(1)) = \mathrm{S^3}/\mathrm{S^1} = \bb C\mathrm{P}^1 \cong S^2.
\end{equation*}

The cotangent bundle of $\mathrm{SU}(2)/\mathrm{S}(\mathrm{U}(1)\times \mathrm{U}(1))$ has a realization as the homogeneous space 
\begin{equation*}
	G/Z_G(H_\Theta)= \mathrm{Sl}(2,\mathbb{C})/\mathrm{S}(\mathrm{Gl}(1,\mathbb{C})\times \mathrm{Gl}(1,\mathbb{C}))
\end{equation*}with

$\mathrm{S}(\mathrm{Gl}(1,\mathbb{C})\times \mathrm{Gl}(1,\mathbb{C}))= \left\lbrace 
\left( \begin{array}{cc}
	z & 0 \\0 & z^{-1} 
\end{array} \right) :  z \in \mathbb{C}^*  
\right\rbrace \cong \mathbb{C}^* := \mathbb{C} \setminus \{ 0 \} $.

In Section \ref{sec:proofmain}, we prove that the cotangent bundle of the $2$-sphere is a symmetric space since the element $H_\Theta$ defined in \eqref{eq:Htheta} is a particular case of the matrix in the equation \eqref{HThetaA}. This guarantees that the homogeneous space $T^*(S^2)=\mathrm{Sl}(2,\mathbb{C})/\mathbb{C}^*$ is a symmetric space.

Furthermore, the following is another realization of the cotangent bundle of $S^2$ as the adjoint orbit of $G=\mathrm{Sl}(2,\mathbb{C})$ in $H_\Theta$
\begin{align*}
G\cdot H_\Theta	=& \left\{ \left( \begin{array}{cc}
		z_1&z_2\\
		z_3&z_4
	\end{array}\right)  H_\Theta \left( \begin{array}{cc}
		z_4&-z_2\\
		-z_3&z_1
	\end{array}\right) : z_1z_4-z_2z_3=1, z_1,z_2,z_3 \in \mathbb{C} \right\} \\
=& \left\{ \dfrac{\sqrt{-1}\pi}{2} \left( \begin{array}{cc}
	1+2z_2z_3&-2z_1z_2\\
	2z_3z_4&-1-2z_3z_2
\end{array}\right): z_1z_4-z_2z_3=1, z_1,z_2,z_3 \in \mathbb{C} \right\} .
\end{align*}

\end{example}	
\begin{remark}
In all the cases of symmetric flag spaces to be considered here, we have symmetric Lie algebras $(\mathfrak{g}, \mathfrak{z}_\Theta, \sigma)$ in the cotangent bundle satisfying
\begin{enumerate}[$(i)$]
	\item $\mathfrak{g}$ is simple and is a vector direct sum $\mathfrak{n}_\Theta^- + \mathfrak{z}_\Theta + \mathfrak{n}_\Theta ^+$ with the relations
	\begin{equation}
		\begin{array}{ccc}
			\left[  \mathfrak{z}_\Theta , \mathfrak{z}_\Theta \right]  \subset  \mathfrak{z}_\Theta, & \left[  \mathfrak{z}_\Theta , \mathfrak{n}_\Theta ^- \right]  \subset \mathfrak{n}_\Theta^-,& \left[  \mathfrak{z}_\Theta , \mathfrak{n}_\Theta^+ \right]  \subset \mathfrak{n}_\Theta^+ , \\ \\ \left[  \mathfrak{n}_\Theta^- , \mathfrak{n}_\Theta^+ \right]  \subset \mathfrak{z}_\Theta , & \left[  \mathfrak{n}_\Theta^- , \mathfrak{n}_\Theta^+\right]   = 0, &\left[  \mathfrak{n}_\Theta^+ , \mathfrak{n}_\Theta^+ \right]   = 0;
		\end{array}
	\end{equation}
	\item the canonical decomposition $\mathfrak{g}= \mathfrak{z}_\Theta \oplus \mathfrak{m}_G$ is given by the direct sum
	\begin{equation*}
		\mathfrak{m}_G = \mathfrak{n}_\Theta^- + \mathfrak{n}^+_\Theta ;
	\end{equation*}
	\item with respect to the Cartan--Killing form $\mathcal{K}$ of $\mathfrak{g}$, the subspaces $\mathfrak{n}_\Theta^\pm$ are dual to each other and, moreover, 
	\begin{equation*}
		\mathcal{K} ( \mathfrak{n}_\Theta^- , \mathfrak{n}_\Theta^+ )=0 \quad \mathrm{and } \quad \mathcal{K} ( \mathfrak{n}_\Theta^+ , \mathfrak{n}_\Theta ^+ )=0.
	\end{equation*}
\end{enumerate}
\end{remark}

\subsection{The dual symmetric space to a symmetric flag space} \label{DSS}

So far, the symmetric spaces considered are Riemannian, considering the Cartan--Killing form as a metric. On the adjoint orbit of $G$, this metric is indefinite and invariant for the symmetries. However, in the flag manifold $\mathbb{F}_\Theta$, it is the opposite of a (definite) Riemannian structure. Moreover, in the \emph{symmetric dual space}, it is an invariant metric by a non-compact Lie group of symmetries. 

In this section, we study the \emph{dual symmetric spaces of the symmetric flag spaces} $(\mathfrak{u}, \mathfrak{u}_\Theta, \sigma)$. Recall that the dual-symmetric Lie algebra is the direct sum (see \cite[Section XI.8]{KN} or \cite[Chapter V.2]{HELGA})
\begin{equation}
	\mathfrak{u}^* = \mathfrak{u}_\Theta + \sqrt{-1}\mathfrak{m},
\end{equation}
with $\sigma^* =\mathrm{Ad}(e^{H_\Theta})\mid _{\mathfrak{u}^*}$. From the expressions \eqref{DSA}, \eqref{Cayley} and \eqref{sodual}, it is going to be clear that $\mathfrak{u}^*$ is isomorphic to a classical real Lie algebra, being the above decomposition the Cartan decomposition of $\mathfrak{u}^*$ with Cartan involution given by $\sigma^*$.

The Example \ref{exemplorealparte2}, to be introduced, is a continuation of Example \ref{exm}, where we deal with $\mathrm{SO}(2)/\{ \pm I \} \cong \mathrm{S^1}$, that is a symmetric space. As we shall see, the cotangent bundle of $\mathrm{S^1}$ has a realization as the one-sheeted hyperboloid \[Q= \{ (x,y,z) \in \mathbb{R}^3 : x^2 + y^2 -z^2 =1 \},\] being also a symmetric space. Moreover, the flag manifold $\mathrm{S^1}$ consists of the intersection of $Q$ with the plane $z=0$. Also, in Example \ref{exemplorealparte2}, the orbit of the exponential of symmetric zero-trace matrices in the matrix\[\left( \begin{array}{cc}
	1&0\\0&-1
\end{array} \right)\] coincides with the curve \[\mathcal{C} = \{ (0,y,z)\in \mathbb{R}^3 : y^2 -z^2 =1, \; y>0 \},\] which is, intuitively, symmetric for the axis $Y$, as we can see in Figure \ref{fig1RAO}.

In the complex case, to generalize the previously described orbits and correspondent curves, we consider the submanifold $\mathcal{S}$: 
\begin{equation*}
	\mathcal{S} : = \mathrm{Ad}(	\mathrm{e}^{\sqrt{-1} \mathfrak{m}} )  H_\Theta \subset G\cdot H_\Theta := \mathrm{Ad}(G)H_\Theta.
\end{equation*}

Observe that, according to Theorem \ref{orbitafibradoSM}, the cotangent bundle to any flag manifold $\bb F_{\Theta}$ can always be realized as the adjoint orbit $\mathrm{Ad}(G)H_\Theta$ for some characteristic element $\Theta$. Hence, the submanifolds $\cal S$ can be conjectured to at least be a homogeneous space.

On the other hand, Proposition \ref{lemfibinterS} below shows that, in general, the intersection of the submanifold $\mathcal{S}$ with the fibers of the bundle $T^*\bb F_{\Theta} \cong \mathrm{Ad}(G)H_\Theta \rightarrow \mathbb{F}_\Theta$ is either null or a single point. Hence, the projection of the cotangent bundle restricted to the submanifold $\mathcal{S}$ is an injective map since each fiber intersecting $\mathcal{S}$ does so in a single element. Therefore, $\mathcal{S}$ continuously deforms into a part of the sphere $\mathbb{F}_\Theta$. 

\begin{theorem} \label{main1}
	Let $\mathbb{F}_\Theta = U/U_\Theta$ be a flag symmetric space with symmetric Lie algebra
	\begin{equation*}
		\mathfrak{u} = \mathfrak{u}_\Theta + \mathfrak{m}.
	\end{equation*}
	Then,
	\begin{enumerate}[$(a)$]
		\item  \label{i}$\mathcal{S}$ is a submanifold of the adjoint orbit $G\cdot H_\Theta$, with $G$ being the complex Lie group with compact real form $U$. Furthermore, 
		\begin{equation*}
			\mathcal{S} \cong U^* /U_\Theta,
		\end{equation*}
		where $U^* / U_\Theta$ is the dual symmetric space of $\mathbb{F}_\Theta$
		\item \label{ii}the submanifold $\mathcal{S}$ is diffeomorphic to a vector space
		\item \label{iii} $\cal S$ is a Riemannian homogeneous space with a Riemannian structure defined by the Cartan--Killing form
		\item \label{iv} let $u\cdot (H_\Theta + \mathfrak{n}_\Theta^+)$ be a fiber in the cotangent bundle of $\mathbb{F}_\Theta$ intersecting the submanifold $\mathcal{S}$, then both spaces are transverse.
	\end{enumerate}
\end{theorem}

\begin{proof} 
	\begin{enumerate}[$(a)$]
		\item The quotient space $U^*/U_\Theta$ is an affine symmetric space by the definition of dual symmetric space. Since $U^*$ is a subgroup of $G$ and $U_\Theta = U\cap Z_\Theta$, then $(U^* , U_\Theta ,\sigma^*)$ is a symmetric subspace of $(G,Z_\Theta ,\sigma)$, with $\sigma^* = \sigma \mid_{U^*}$. According to \cite[Theorem 4.1, p. 234]{KN}, $U^*/U_\Theta$ is a totally geodesic submanifold of $G/Z_\Theta$ (concerning the canonical connection of $G/Z_\Theta$) whose canonical connection is the restriction of the canonical connection of $G/Z_\Theta$ to $U^*/U_\Theta $. 	
		
		We now apply Theorem \ref{orbitafibradoSM} to the non-compact connected Lie group $U^*$, obtaining that $U^*/U_\Theta$ is diffeomorphic to $\exp (\sqrt{-1} \mathfrak{m})$, which is simply connected: 
		\begin{equation}
			\mathrm{Ad}(U^*) H_\Theta = \mathrm{Ad}(e^ {\sqrt{-1}\mathfrak{m}}) \mathrm{Ad}(U_\Theta) H_\Theta=\mathrm{Ad}(e^{\sqrt{-1}\mathfrak{m}})H_\Theta =: \mathcal{S}.
		\end{equation} 
		\item According to the previous item, $\mathcal{S}$ is diffeomorphic to $\exp (\sqrt{-1}\mathfrak{m})$, thus, it is diffeomorphic to the vector space $\sqrt{-1} \mathfrak{m}$.
		\item The tangent space of the dual symmetric space $U^*/U_\Theta$ at the origin is isomorphic to $\sqrt{-1}\mathfrak{m}$ and the Cartan--Killing form in $\sqrt{-1}\mathfrak{m}$ is  $\mathcal{K}(\sqrt{-1}X,\sqrt{-1}Y)=-\mathcal{K}(X,Y) =: \langle X,Y \rangle$ for $X,Y \in \mathfrak{m}$. Since $\mathfrak{m}\subset \mathfrak{u}$, then the Cartan--Killing form in $\sqrt{-1}\mathfrak{m}$ is positive defined. %Let $\mathfrak{u}^*$ be the noncompact Lie algebra of $U^*$ with Cartan decomposition $\mathfrak{u}^* = \mathfrak{u}_\Theta \oplus \sqrt{-1}\mathfrak{m}$. 
		The homogeneous space $U^*/U_\Theta$ admits invariant metrics since $U_\Theta$ is compact. A natural metric in $U^*/U_\Theta$ is that from the restriction of the Cartan--Killing form to $\sqrt{-1}\mathfrak{m}$. Since $\mathfrak{u}^*$ is simple, then the adjoint representation in $\sqrt{-1} \mathfrak{m}$ is irreducible, and the invariant metric is essentially unique, defined by the Cartan--Killing form.
		\item On the one hand, every fiber has the form $u\cdot (H_\Theta + \mathfrak{n}_\Theta^+) = u \cdot (N_\Theta^+ \cdot H_\Theta)$, for $u\in U$. To say that $\mathcal{S}$ is transverse to $H_\Theta + \mathfrak{n}_\Theta^+$  is equivalent to proving that for every $x\in \mathcal{S} \cap u\cdot (H_\Theta + \mathfrak{n}_\Theta^+)$ it must holds
		\begin{equation*}
			T_x (\mathcal{S} ) + T_x(u \cdot (H_\Theta + \mathfrak{n}_\Theta^+)) = T_x (G\cdot H_\Theta),
		\end{equation*}
		for each $u\in U$ such that $\mathcal{S} \cap u\cdot (H_\Theta + \mathfrak{n}_\Theta^+) \neq \emptyset$.
		
		On the other hand, the Lie algebra $\mathfrak{g}$ can be seen as the set of left-invariant fields. Equation
		\begin{equation*}
			\dfrac{d}{dt} \left(\mathrm{Ad}(e^{tX}(Y))\right) \mid_{t=0} = \left[ X, Y \right] , \qquad X,Y\in \mathfrak{g},
		\end{equation*}
		ensures that the tangent space of $\mathcal{S}$ at $e^A \cdot H_\Theta = u \cdot (H_\Theta + X) = u e^Z \cdot H_\Theta$, with $A\in \sqrt{-1}\mathfrak{m} , X,Z \in \mathfrak{n}_\Theta^+$, $u\in U$ and $e^Z \cdot H_\Theta = H_\Theta + X$, is precisely $[e^A \cdot H_\Theta , \sqrt{-1} \mathfrak{m}]$. Moreover, the tangent space to the fiber $u\cdot (H_\Theta + \mathfrak{n}_\Theta^+)$ at $e^A \cdot H_\Theta $ coincides with $[ e^A \cdot H_\Theta, \mathfrak{n}_\Theta^+ ]$. In this manner, the sum of the former two tangent spaces is
		\begin{equation}
			[e^A \cdot H_\Theta , \sqrt{-1} \mathfrak{m}] + [e^A \cdot H_\Theta , \mathfrak{n}_\Theta^+ ]  = [e^A \cdot H_\Theta , \sqrt{-1} \mathfrak{m} + \mathfrak{n}_\Theta^+ ] = [ e^A \cdot H_\Theta , \mathfrak{m}_G ],	
		\end{equation}
		i.e., it is the tangent space to the adjoint orbit $G\cdot H_\Theta$ at $e^A \cdot H_\Theta$. \qedhere
	\end{enumerate}
\end{proof}

\begin{proposition}	\label{lemfibinterS}
	The intersection of $\mathcal{S}$ with the fiber $  H_\Theta + \mathfrak{n}_\Theta^+ $ is the point set $ \{ H_\Theta \}$.
\end{proposition}
\begin{proof}
	In Section \ref{sec:proofmain}, we shall compute the intersection of the submanifold $\mathcal{S}$ with the fiber $H_\Theta + \mathfrak{n}_\Theta^+$ in the three cases of symmetric flag spaces. The results will be evidenced in \eqref{interfibA}, \eqref{interfibC}, and \eqref{interfibD}. 
\end{proof}

The symmetry $\tilde{ \sigma}$ in the submanifold $\mathcal{S}$ is given by  
\begin{equation}
	\begin{array}{clcl}
		\tilde{\sigma} :& \mathcal{S} & \rightarrow & \mathcal{S} \\
		\quad & e^{A} \cdot H_\Theta & \mapsto & \mathrm{C}_{e^{H_\Theta}} (e^A) \cdot H_\Theta,
	\end{array}
\end{equation} 
for $A\in \sqrt{-1} \mathfrak{m}$. Then, $ \tilde{ \sigma} (e^A 
 \cdot H_\Theta )  = \exp (\mathrm{Ad}(e^{H_\Theta} A) \cdot H_\Theta) = \exp (\sigma (A))\cdot H_\Theta  = \exp (-A) \cdot H_\Theta$, since $(\mathfrak{u}^* , \mathfrak{u}_\Theta, \sigma)$, $\sigma=$Ad($e^{H_\Theta})$ is a symmetric Lie algebra. Once more, we reinforce using the same notation $\sigma$ for the automorphism in the Lie group and algebra. We prove:
 
\begin{theorem} \label{main2}
	Suppose that, for some $u\in U$, the fiber $u \cdot (H_\Theta  + \mathfrak{n}^+_\Theta)$ intersects the submanifold $\mathcal{S}$. Then, the fiber $\sigma (u)\cdot (H_\Theta + \mathfrak{n}_\Theta^+)$ intersects $\mathcal{S}$.
\end{theorem}
\begin{proof}
Each element $e^A  \cdot H_\Theta$ of $\mathcal{S}$ is an element of some fiber of the cotangent bundle and has the form $u\cdot ( H_\Theta + X ) = ue^Z \cdot H_\Theta$, for some $u\in U$ and $Z, X\in \mathfrak{n}_\Theta^+$, since $\mathcal{S} \subset G\cdot H_\Theta$. Hence,
\begin{align*}
\tilde{ \sigma } (e^A \cdot H_\Theta) =& \tilde{ \sigma } (u e^Z \cdot H_\Theta) \nonumber \\
	e^{-A} \cdot H_\Theta =& \sigma (u) \sigma (e^Z) \cdot H_\Theta \nonumber \\
	\qquad =& \sigma (u) e^{-Z} \cdot H_\Theta,\quad Z\in \mathfrak{m}_G \nonumber \\
	\qquad \in & \sigma (u) \cdot (H_\Theta + \mathfrak{n}^+_\Theta). \qedhere
\end{align*} \end{proof}

However, not every fiber intersects $\mathcal{S}$. The following theorem indicates some fibers that do not intersect $\mathcal{S}$.

\begin{theorem}
	The fiber $u\cdot (H_\Theta + \mathfrak{n}_\Theta^+)$ does not intersect $\mathcal{S}$ if $u\cdot \mathfrak{n}_\Theta^+ \subset \mathfrak{u}^*$ and $u\cdot H_\Theta \notin \mathfrak{u}_\Theta$.
\end{theorem}
\begin{proof}
 Let $u\cdot (H_\Theta + \mathfrak{n}_\Theta^+)$ a fiber such that $u\cdot \mathfrak{n}_\Theta^+  \subset \mathfrak{u}^*$ and $u \cdot H_\Theta \notin \mathfrak{u}_\Theta$. Suppose that $u\cdot (H_\Theta + \mathfrak{n}_\Theta ^+) \cap \mathcal{S} \neq \varnothing.$ The fiber is contained in $u\cdot H_\Theta + \mathfrak{u}^*$. If $u \cdot( H_\Theta + X ) \in \mathcal{S} \subset \mathfrak{u}^*$, with $X\in \mathfrak{n}_\Theta^+$, then $u\cdot H_\Theta \in \mathfrak{u}^* \cap \mathfrak{u} = \mathfrak{u}_\Theta$ and this is a contradiction.  Hence
\begin{equation*}
	(u\cdot (H_\Theta + \mathfrak{n}_\Theta^+)) \cap \mathcal{S} = \varnothing. \qedhere
\end{equation*}
\end{proof}

\begin{theorem}
	Let	$u\cdot (H_\Theta + X) \in \mathcal{S} \subset T^*\mathbb{F}_\Theta$ be an element  that  satisfies $u\cdot \mathfrak{n}_\Theta^+ \cap \mathfrak{u}^* = \{ 0 \}$. Then, the  intersection of the submanifold $\mathcal{S}$ with the fiber $ u\cdot ( H_\Theta + \mathfrak{n}_\Theta^+ )$ is $\{ u\cdot (H_\Theta + X) \}$.
\end{theorem}
\begin{proof}

	Consider an element $u \cdot (H_\Theta + Y)$ in the intersection $\mathcal{S} \cap u\cdot (H_\Theta + \mathfrak{n}_\Theta^+)$. Then $u \cdot (H_\Theta + Y) - u \cdot (H_\Theta + X) = u\cdot (Y-X) \in u\cdot \mathfrak{n}_\Theta^+ \cap \mathfrak{u}^*   = \{ 0 \}$, since $\mathcal{S} \subset \mathfrak{u}^*$. Hence $X=Y$. 
\end{proof}

\begin{example}[Continuation of Example \ref{exm}]	\label{exemplorealparte2} 
	
		The symmetric part of the Cartan decomposition of $\mathfrak{sl}(2,\mathbb{R})$ is generated by the matrices $A$ and $B$ (see equation \eqref{basedesl}). Thus, the submanifold $\mathcal{S}$ is represented by 
	\begin{align}
		\mathrm{Ad}(e^{tA})B=:& \exp (tA)\cdot B \nonumber \\
		=& \left\{ \left( \begin{array}{cc}
			\cosh t &  \sinh t \\ \sinh t & \cosh t
		\end{array} \right) \cdot B = 0 A + \cosh (2t) B - \sinh (2t) C : t\in \mathbb{R} \right\} \nonumber \\
		=& \{ (0,y,z) \in \mathbb{R}^3 : y^2-z^2=1, \; y>0 \} \label{hiperbola}
	\end{align}
	and it is the intersection of the hyperboloid $x^2+y^2-z^2=1$ with the semi-plane $\{ x=0, y>0 \}$ contained in the vector space $\mathrm{span}_\mathbb{R} \{ B,C \}$. Here, the submanifold $\mathcal{S}$ is symmetric for the abelian vector subspace of $\mathfrak{sl}(2,\mathbb{R})$ generated by $B$. Hence, we have three symmetric spaces: The hyperboloid, given by equation \eqref{hyperboloide}, the $1$-sphere, and a branch of the hyperbola \eqref{hiperbola}, contained in the plane $YZ$.
	
	\begin{figure}[htb]
		\centering
		\includegraphics[width=14.5cm]{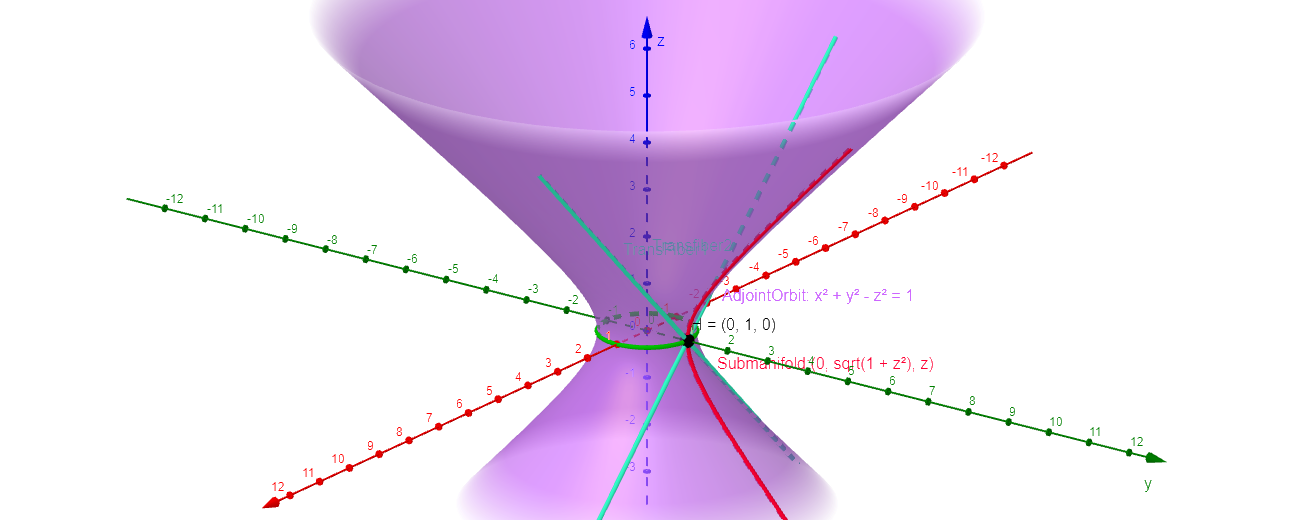}
		\caption{Real adjoint orbit.}
		\label{fig1RAO}
	\end{figure}
	
	The intersection of $\mathcal{S}$ with the fiber $k\cdot (B+ \mathfrak{n})$ is the set of the vectors $xA+yB+zC$ in \eqref{fiberexmreal}, such that $x=0$ and $y>0$. These conditions are equivalent to the system
	\begin{equation}
		\left\{ \begin{array}{c}
			-2\sin (2t) + r \cos (2t) =0 \\
			2	\cos (2t) + r \sin (2t) >0
		\end{array} \right. , t\in [0,\pi] .
	\end{equation}
	This system has solution only for $t\in ( -\frac{\pi}{4} ,\frac{\pi}{4} )$. Thus, the only fibers intersecting $\mathcal{S}$ are $k\cdot (B+ \mathfrak{n})$ with $k$ as in \eqref{kexm} and $t\in ( -\frac{\pi}{4} ,\frac{\pi}{4} )$. The intersection is the only vector satisfying the equation \eqref{elementodelafibra} with $r=2\tan (2t)$. Furthermore, note that when  $k$ is equal to  
	\begin{equation}
		\frac{1}{\sqrt{2}}	\left( \begin{array}{cc}
			1 & 1 \\ -1 & 1
		\end{array} \right) \; \mathrm{or} \; \frac{1}{\sqrt{2}}	\left( \begin{array}{cc}
			1 & -1 \\ 1 & 1
		\end{array} \right) ,
	\end{equation} then $k\cdot B$ is equal to $-A$ and $A$, respectively. Hence, the fibers passing by $\pm A$ do not intersect the submanifold $\mathcal{S}$.
	
	We can see these results in Figure \ref{fig2FAS}. All the lines in this figure are fibers, and the sky blue lines are the fibers passing by $\pm A$. We can see how the fibers passing by the points in the (green) semicircle of radius one contained in $\{ (x,y,0): x\in \mathbb{R}, y>0 \}$ intersect the submanifold $\mathcal{S}$ represented by the red curve.
	
	\begin{figure}[!h]	\label{Fibrasfigura}
		\centering
		\includegraphics[width=8cm, height=7.5cm]{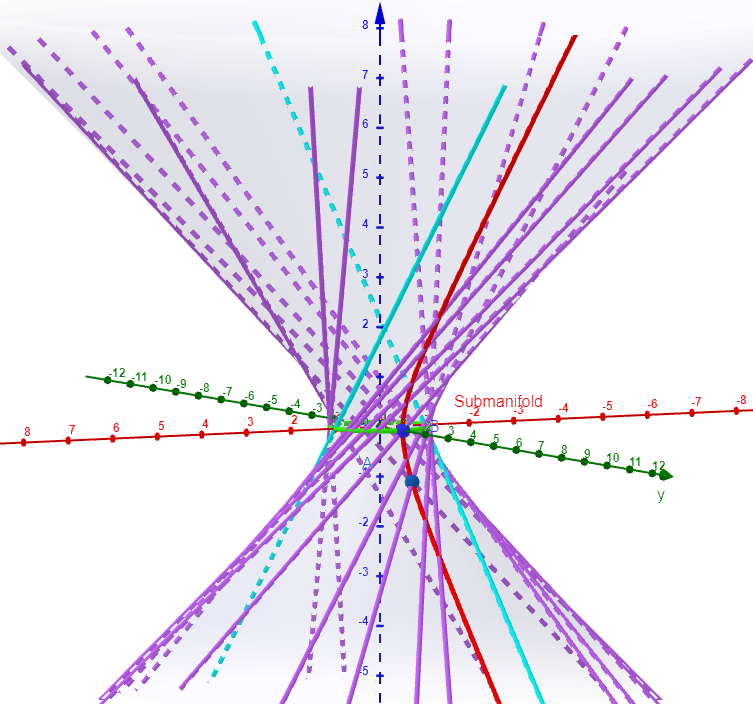}
		\caption{Fibers of $T^*\mathrm{S^1}$ and the submanifold $\mathcal{S}$.}
		\label{fig2FAS}
	\end{figure}
\end{example}
\pagebreak

\begin{example}[Continuation of Example \ref{Examplecomplexo}]
	
The manifold $\mathcal{S} :=(\exp\mathfrak{s})\cdot H_\Theta$, with \[\mathfrak{s}= \sqrt{-1}\mathfrak{m} = \left\lbrace \left( \begin{array}{cc}
		0&b\\ \bar{b} &0
	\end{array} \right) : b\in \mathbb{C} \right\rbrace,\] is 
	\begin{eqnarray*}
		\mathcal{S}=	e^{\sqrt{-1}\mathfrak{m}}  \cdot H_\Theta = &\left\lbrace  \dfrac{\sqrt{-1}\pi}{2} \left( \begin{array}{cc}
			\cosh (2 \left|  b \right| ) & - \frac{b}{\left| b\right| } \sinh (2 \left| b \right| ) \\
			\frac{\bar{b}}{\left|  b \right| } \sinh (2 \left| b\right| ) & - \cosh (2 \left| b \right| ) 
		\end{array} \right) : b\in \mathbb{C} \right\rbrace \label{ESE} 
		\end{eqnarray*}
	The diagonal entries of the matrices in $\mathcal{S}$ are complex numbers with real parts equal to zero.
		
		The fibers of the cotangent bundle are $u\cdot (H_\Theta + \mathfrak{n}^+)$, with $u = \left( \begin{array}{cc}
			\alpha & \beta \\ -\bar{\beta} &\bar{\alpha}
		\end{array} \right) \in \mathrm{SU}(2)$ and each fiber is the set
		\begin{equation*}
			\left\lbrace \left( \begin{array}{cc}
				\frac{\sqrt{-1}\pi}{2} (2 \left|  \alpha \right| ^2 -1) + x\alpha \bar{\beta} & -\sqrt{-1}\pi \alpha \beta + x\alpha^2 \\
				-\sqrt{-1}\pi \bar{\alpha} \bar{\beta}-x\bar{\beta}^2 & -\frac{\sqrt{-1}\pi}{2}(2 \left|  \alpha \right| ^2 -1)-x\alpha \bar{\beta}
			\end{array} \right) : x\in \mathbb{C} \right\rbrace .
		\end{equation*}
		To find out which fibers intersect $\mathcal{S}$, we are looking for elements in $u\cdot (H_\Theta + \mathfrak{n}^+)$ such that for each $\alpha$ and $\beta$ as above, we must find if there exists $x\in \mathbb{C}$ satisfying the following conditions:
		\begin{align}
			&	\bar{x}\bar{\alpha}\beta = -x\alpha \bar{\beta} \label{cond1} \\
			&	\sqrt{-1}\pi \bar{\alpha}\bar{\beta} + \bar{x}\bar{\alpha}^2 = -\sqrt{-1}\pi \bar{\alpha}\bar{\beta} - x\bar{\beta}^2 	\label{cond2} \\
			&	\frac{\pi}{2}(2\left| \alpha \right| ^2 -1) + \mathrm{Im} (x\alpha \bar{\beta}) \geq \frac{\pi}{2} 	\label{cond3}
		\end{align}
		
		The first and second conditions come from the form of the matrices in $\mathfrak{su}(1,1)$. The third condition is because, from \eqref{ESE}, the imaginary part of the first element in the diagonal of matrices in $\mathcal{S}$ is greater than $\frac{\pi}{2}$.
		
		If $\beta=0$, then $u\in \mathrm{S}(\mathrm{U}(1)\times \mathrm{U}(1))$ and $\alpha \neq 0$, thus the fiber is $(H_\Theta + \mathfrak{n}^+)$. The unique $x\in \mathbb{C}$ that satisfies the three conditions above is $x=0$. Hence, the intersection $(H_\Theta+\mathfrak{n}^+)\cap \mathcal{S}$ is $\{ H_\Theta \}$.
		
		If $\alpha =0$, then $ \left| \beta \right| =1$ and the fiber is $(H_\Theta + \mathfrak{n}^-)$. The equation \eqref{cond2} tells us that $x=0$. Putting $x=0$ in \eqref{cond3}, we have a contradiction. Hence, the intersection $(H_\Theta +\mathfrak{n}^-) \cap \mathcal{S}$ is empty.
		
		If $\alpha \neq 0 \neq \beta$, from equations \eqref{cond1} and \eqref{cond2}, we have the new condition
		\begin{equation*}
			2\sqrt{-1}\pi \bar{\alpha } \beta = x (2 \left|  \alpha \right| ^2 - 1).
		\end{equation*}

	Here, we have two cases:
		\begin{itemize}
			\item[a)] If $2 \left|  \alpha \right| ^2 - 1 = 0$, then $2\sqrt{-1}\pi \bar{\alpha}\beta = 0$. This is a contradiction, since $\bar{\alpha} \beta \neq 0 $. Hence, the fibers $u\cdot (H_\Theta + \mathfrak{n}^+)$ with
			\begin{equation*}
				u\in \left\lbrace \left( \begin{array}{cc}
					\alpha & \beta \\- \bar{\beta} & \bar{\alpha}
				\end{array} \right) :  \lvert \alpha \rvert ^2= \lvert \beta \rvert ^2 = \frac{1}{2} , \alpha , \beta \in \mathbb{C} 
				\right\rbrace ,
			\end{equation*} has empty intersection with the manifold $\mathcal{S}$.
			
			\item[b)] If $2 \left|  \alpha \right| ^2 - 1 \neq 0$, from equations \eqref{cond1} and \eqref{cond2}, we find the complex number
			\begin{equation*}
				x= \frac{2\pi\sqrt{-1}}{2 \left|  \alpha \right| ^2 - 1} \bar{\alpha} \beta .
			\end{equation*}
		Condition \eqref{cond3}, with $x$ as above, leads to the following condition
			\begin{equation*}
				\frac{1}{2} <  \left|  \alpha \right| ^2  < 1 .
			\end{equation*}
		\end{itemize}

		Hence, $(u\cdot (H_\Theta +\mathfrak{n}^+) ) \cap \mathcal{S}$ is non-empty if 
		\begin{equation}	\label{notempexm}
			u\in \left\lbrace \left( \begin{array}{cc}
				\alpha & \beta \\- \bar{\beta} & \bar{\alpha}
			\end{array} \right) : \frac{1}{2} < \lvert \alpha \rvert ^2 \leq 1, \lvert \alpha \rvert ^2 +  \lvert \beta \rvert ^2 = 1 , \alpha , \beta \in \mathbb{C} 
			\right\rbrace.
		\end{equation}
		Furthermore, 
		\begin{equation*}
			(u\cdot (H_\Theta +\mathfrak{n}^+) ) \cap \mathcal{S} = \left\lbrace \frac{\sqrt{-1}\pi}{2 \left|  \alpha \right| ^2 - 1} \left( \begin{array}{cc}
				1 & \alpha \beta \\- \bar{\alpha} \bar{\beta} & -1
			\end{array} \right) 
			\right\rbrace.
		\end{equation*}
		
		Considering the following basis for $\mathfrak{u}$:
		\begin{equation*}
			e_1 = \left( \begin{array}{cc}
				0&1\\1&0
			\end{array} \right) ,\; e_2 = \left( \begin{array}{cc}
				\sqrt{-1}&0\\0&-\sqrt{-1}
			\end{array} \right) , \; e_3 = \left( \begin{array}{cc}
				0& \sqrt{-1}\\ \sqrt{-1}&0
			\end{array} \right) ,
		\end{equation*} 
		the matrices in $u\cdot H_\Theta$, for $u$ as in equation \eqref{notempexm}, have the decomposition
		\begin{equation*}
			\pi \mathrm{Im}(\alpha \beta) e_1 + \frac{\pi}{2} (2 \left|  \alpha \right| ^2 - 1) e_2 - \pi \mathrm{Re}(\alpha \beta) e_3, 
		\end{equation*} with the second component greater than zero and 
		\begin{equation*}
			\pi^2 [ \mathrm{Im}(\alpha \beta) ]^2 + \frac{\pi^2}{4} (2 \left|  \alpha \right| ^2 - 1)^2 + \pi ^2 [ \mathrm{Re}(\alpha \beta) ]^2 = \frac{\pi^2}{4}.
		\end{equation*}
		Hence, $u\cdot H_\Theta$, with $u$ as in \eqref{notempexm}, is the semi-sphere $\{ (a,b,c) \in \mathfrak{su}(2) : a^2+b^2+c^2= \frac{\pi^2}{4},\; b>0 \}$. We then conclude that $\mathcal{S}$ is projected into the flag as the semi-sphere through the fibers.
		
		Since the manifold $\mathcal{S}$ is contained in \[\mathfrak{su}(1,1) = \left\lbrace \left( \begin{array}{cc}
			\alpha & \beta \\ \bar{\beta} & -\alpha
		\end{array} \right) : \beta \in \mathbb{C}, \bar{\alpha}= -\alpha  \right\rbrace,\] we can express each of its elements as a combination of the matrices 
	\begin{equation*}
		 A= \left(  \begin{array}{cc}
			0&1\\1&0
		\end{array} \right) , B= \left( \begin{array}{cc}
			\sqrt{-1}&0\\ 0&-\sqrt{-1}
		\end{array}\right) , C= \left( \begin{array}{cc}
			0&\sqrt{-1}\\-\sqrt{-1}&0
		\end{array} \right).
	\end{equation*}
	 In that base, the manifold $\mathcal{S}$  is the part of the two-sheeted hyperboloid $x^2 + z^2 -y^2 = -\frac{\pi^2}{4}$ with positive $Y$-component:
		\begin{equation*}
			\mathcal{S} = \left\lbrace  \frac{\pi}{2} \sin t \sinh(2r) \cdot A + \frac{\pi}{2} \cosh (2r) \cdot B - \frac{\pi}{2} \cos t \sinh(2r) \cdot C : t\in [0,2\pi] , r \geq 0 \right\rbrace .
		\end{equation*}
	
		Observe that the set of matrices $xA + yB + zC \in \mathfrak{su}(1,1)$ with determinant equal to $\mathrm{det}(H_\Theta)= \frac{\pi^2}{4}$ is the hyperboloid $x^2 + z^2 - y^2 = -\frac{\pi^2}{4}$ and has $\left| \mathcal{W} \right| $ connected components where $\mathcal{W}$ is the Weyl group of $\mathfrak{sl}(2,\mathbb{R})$. Also, the flag is maximal in this example, and $H_\Theta$ is a regular element. 
		
		Let $w_0 \in \mathcal{W}$ be the principal involution, then $w_0 \cdot H_\Theta = -H_\Theta$, and the adjoint orbit $\mathrm{Ad}(\mathrm{SU}(1,1)) (-H_\Theta) = -\mathrm{Ad}(\mathrm{SU}(1,1)) H_\Theta$ is the sheet of the hyperboloid $x^2 + z^2 - y^2 = -\frac{\pi^2}{4}$ such that $y<0$. Thus,
		\begin{align*}
			\{ M \in \mathfrak{sl}(2,\mathbb{R}): \mathrm{det}(M)= \frac{\pi^2}{4} \} =& \{  (x,y,z) \in \mathbb{R}^3 : x^2 + z^2 - y^2 = -\frac{\pi^2}{4} \} \\
			 =& \mathrm{Ad}(\mathrm{SU}(1,1)) H_\Theta \cup \mathrm{Ad}(\mathrm{SU}(1,1)) w_0H_\Theta.
		\end{align*}
\end{example}

\section{Symplectic forms and Lagrangian submanifolds}

Recall that the cotangent bundle $T^*M$ of a manifold $M$ is naturally a symplectic manifold: one can always define a \emph{canonical symplectic structure} on $T^*M$ given $d\theta$, where $\theta$ is the \emph{tautologically} $1$-form $\theta$ on $T^*M$.

As Theorem \ref{main1} has proved, the cotangent bundle to a symmetric flag space also happens to be symmetric. Moreover, once we are dealing with semisimple Lie algebras, let us define a $2$-form in $\mathfrak{g}$ by
\begin{equation}	\label{hermitianform}
	\mathcal{H}_\sigma(X,Y) = -\mathcal{K}(X,\sigma Y) \qquad X,Y\in \mathfrak{g}.
\end{equation}

Equation \eqref{hermitianform} defines a hermitian form in $\mathfrak{g}$, thus defining a symplectic form in $\mathfrak{g}$:
\begin{equation}	\label{symplecticform}
	\omega (X,Y) := -\mathrm{Im} \mathcal{H}_\sigma (X,Y) \qquad X,Y \in \mathfrak{g}
\end{equation}

Since Theorem \ref{orbitafibradoSM} ensures that symmetric flag space $\mathbb{F}_\Theta$ can be considered as the adjoint orbit $\mathrm{Ad}(U)\cdot H_\Theta$ of the compact Lie group $U$, contained in $\mathfrak{u}$, the conjugation $\sigma$ restricts to $\mathfrak{u}$ as the identity function. Then, the symplectic form above in $\mathrm{Ad}(U)\cdot H_\Theta$ is 
\begin{equation}
	\omega (X,Y) = -\mathrm{Im}\mathcal{K}(X,Y)\qquad X,Y \in \mathfrak{u}.
\end{equation}
 
Moreover, item \eqref{item2lino} of Theorem \ref{orbitafibradoSM} ensures that this $2$-form induces a symplectic form in the adjoint orbit through the pullback of $\omega$ of the inclusion $G\cdot H_\Theta \hookrightarrow \mathfrak{g}$ and with $T_{g\cdot H_\Theta} G\cdot H_\Theta= [\mathfrak{g}, g\cdot H_\Theta] $, for each $g\in G$. Such a $2$-form is what is called \emph{the real Kostant-Kirillov-Souriau} $\Omega $ (KKS) in the coadjoint orbit $\mathrm{Ad}^*(G) \alpha$, given by 
\begin{equation}
	\Omega_\alpha ( \tilde{X} (\alpha) , \tilde{Y}(\alpha) ) = \alpha [X,Y] \qquad X,Y\in \mathfrak{g}.
\end{equation}
As we already mentioned, whenever $\lie g$ is semisimple, there is a natural correspondence induced by the Cartan--Killing form between the coadjoint and the adjoint orbit, which justifies $\Omega$ being a symplectic form in the adjoint orbit.

We prove:
\begin{theorem}
	The symmetric flag space $\mathbb{F}_\Theta = U/U_\Theta$ and its dual symmetric space $\mathcal{S}$ are Lagrangian submanifolds of the adjoint orbit $G\cdot H_\Theta$ with the symplectic form defined in \eqref{symplecticform}.
\end{theorem}
\begin{proof}
Note that the Cartan--Killing form is a symmetric $2$-form while $\omega$ is anti-symmetric. Moreover, since the Lie algebra $\mathfrak{u}$ is real, the symplectic form $\omega$ restricted to the submanifold $\mathbb{F}_\Theta$ is zero.

Now, let us see which is the restriction of the symplectic form above to the submanifold $\exp (\sqrt{-1}\mathfrak{m})\cdot H_\Theta = U^* \cdot H_\Theta$ of $G\cdot H_\Theta$. This submanifold is contained in $\mathfrak{u}^*$ and each tangent space $T_{v\cdot H_\Theta} U^* \cdot H_\Theta =  [\mathfrak{u}^* , v\cdot H_\Theta]$, for each $v\in \exp (\sqrt{-1}\mathfrak{m})$, is also contained in $\mathfrak{u}^*$. The real Lie algebra $\mathfrak{u}^*$ is invariant by the conjugation $\sigma$. The Cartan--Killing form restricted to $\mathfrak{u}^*$ coincides with the Cartan--Killing form of $\mathfrak{u}^*$ since it is a real form. Then, the symplectic form restricted to $\mathfrak{u}^*$ is zero. \qedhere
\end{proof}

On the other hand, the Kostant-Kirillov-Souriau symplectic form (KKS) restricted to $\mathcal{S}$ is 
\begin{equation}	\label{KKSform}
	\Omega _\xi ( [X,\xi ]  , [Y,\xi] ) = \mathcal{K} ( \xi , [X,Y] ), \qquad X,Y \in  \mathfrak{u}^*, \xi \in \mathcal{S} .
\end{equation}

Hence, if $\xi = H_\Theta$ and $X = \sqrt{-1}A_\alpha , Y= \sqrt{-1}Z_\alpha$, for some $\alpha \in \Pi^+ \setminus \langle \Theta \rangle ^+$, then
\begin{align*}
	\Omega _{H_\Theta} ([\sqrt{-1}A_\alpha , H_\Theta] , [\sqrt{-1}Z_\alpha , H_\Theta] ) =& \mathcal{K} (H_\Theta , [\sqrt{-1}A_\alpha , \sqrt{-1}Z_\alpha]) \\
	=& \mathcal{K} (H_\Theta , -[A_\alpha , Z_\alpha]) \\
	=& \mathcal{K} (H_\Theta , 2\sqrt{-1} H_\alpha) \\
	=& 2\sqrt{-1} \alpha (H_\Theta) \\
	\neq & 0.
\end{align*}
Therefore, the submanifold $\mathcal{S}$ is non-Lagrangian for the KKS form. We thus have proved:
\begin{corollary}
    The diffeomorphism $\iota : (\mathrm{Ad}(G)\cdot H_\Theta,\Omega) \rightarrow (T^*\mathbb{F}_\Theta,\omega)$ is not a symplectomorphism.
\end{corollary}

\section{On the geometry of symmetric spaces and associated bundles: geodesics and curvature}
On the one hand, according to \cite[Chapter 3, p.230-p.234]{KN}, we know that to each symmetric space $G/H$, one can guarantee the existence of an invariant affine connection which recovers a principal bundle over $G/H$. Moreover, there is a one-to-one correspondence between the set of $G$-invariant connections in $G$ (connection in the bundle $G(M,H)$ which is invariant by the left translations of $G$) and the set of linear mappings $\Lambda_\mathfrak{m} : \mathfrak{m} \rightarrow \mathfrak{h}$ such that
 	\begin{equation*}
 		\Lambda_\mathfrak{m} (\mathrm{Ad}(h)X ) = \mathrm{Ad}(h) (\Lambda_\mathfrak{m} (X) ), \; X\in \mathfrak{m}, h\in H.
 	\end{equation*}

 If $\Lambda:\mathfrak{g}\rightarrow \mathfrak{g}$ is defined by
 	\begin{equation*}
 		\Lambda (X)= \left\lbrace 
 		\begin{array}{cc}
 			X, & \; X \in \mathfrak{h},\\
 			\Lambda_\mathfrak{m} (X), & \; X\in \mathfrak{m}
 		\end{array} \right.
 	\end{equation*}
 	one indetifies any $G$-invariant connection in $G$ with a connection form $\omega$ as
    \begin{equation*}
    	\Lambda(X)=\omega_{u_0} (\tilde{X}) \qquad X\in \mathfrak{g},
    \end{equation*} where $\tilde{X}$ denotes the natural lift to $G$ of a vector field $X\in \mathfrak{g}$ of $M=G/H$ and $u_0\in G$ fixed. The invariant connection corresponding to $\Lambda_\mathfrak{m} =0$ is called \emph{canonical connection} of $(G,H,\sigma)$ or $G/H$. 
 
On the other hand, we recall that a homogeneous space $M=G/H$ with a $G$-invariant indefinite Riemannian metric is said to be \emph{naturally reductive} if it admits an $\mathrm{Ad}(H)$-invariant decomposition $\mathfrak{g}=\mathfrak{h}\oplus \mathfrak{m}$ satisfying the condition 
\begin{equation}
	B(X,[Z,Y]_\mathfrak{m}) + B([Z,X]_\mathfrak{m},Y)=0, \qquad \forall X,Y,Z \in \mathfrak{m}
\end{equation}
where $B(\cdot, \cdot)$ is an $\mathrm{Ad}(H)$-invariant non-degenerate symmetric bilinear form on $\mathfrak{m}$. We observe that \emph{every symmetric space is naturally reductive} since $[X, Y]_\mathfrak{m}=0$ for all $X, Y \in \mathfrak{m}$.

Once every reductive homogeneous space $M=G/H$ admits a \emph{unique} torsion-free $G$-invariant affine connection (\cite{KN}), having the same geodesics as the canonical connection, which is obtained for
 	\begin{equation}
 		\Lambda_\mathfrak{m}(X)(Y) = \dfrac{1}{2}[X,Y]_\mathfrak{m}, \quad X,Y \in \mathfrak{m},
 	\end{equation}
we say that such an invariant connection is the \emph{natural torsion-free connection} on $G/H$ (with respect to the decomposition $\mathfrak{g}=\mathfrak{h} \oplus \mathfrak{m}$). 

Recall that if $(G, H,\sigma)$ is a symmetric space with $G$ semisimple, the restriction of the Cartan--Killing form $\mathcal{K}(\cdot, \cdot)$ of $\mathfrak{g}$ to $\mathfrak{m}$ defines a $G$-invariant (indefinite) Riemannian metric on $G/H$, which induces the canonical connection on $G/H$. Moreover, for the canonical connection of a symmetric space $(G, H,\sigma)$, the homogeneous space $M=G/H$ with origin $o$ is a (complete) affine symmetric space with symmetries $s_x$ such that, for each $X\in \mathfrak{m}$, the curve $f_t =\exp tX$ in $G$ implies that the induced curve $x_t = f_t  o$ in $M =G/H$ is a geodesic. Conversely, every geodesic in $M$ starting from $o$ is $f_t \cdot o=\exp tX \cdot o$ for some $X\in \mathfrak{m}$.

\subsection{Geodesics on symmetric flag spaces and their cotangent spaces}
\label{sec:geodesics}
Considering the essential aspects of the former section, let us describe some geodesics of every studied flag symmetric space that appears in this paper.
To do so, recall that the tangent vector space to the adjoint orbit at the origin can be expressed in the following ways (direct sums):
\begin{equation}	\label{decompoespvectoriales}
	\mathfrak{m}_G = \mathfrak{m} + \sqrt{-1}\mathfrak{m} = \mathfrak{m} + \mathfrak{n}_\Theta^+.	
\end{equation}

Hence, every  $X\in \mathfrak{m}_G$ admits the decompositions
\begin{eqnarray}	
	X= X_m + X_s, & \qquad X_m\in \mathfrak{m},\; X_s \in \sqrt{-1}\mathfrak{m} \label{decompoenm} \\
	X= X_m + X_f,  & \qquad X_m \in \mathfrak{m}, \; X_f \in \mathfrak{n}_\Theta^+ .  \label{decompoenfib}
\end{eqnarray}

With a canonical connection given by the Cartan--Killing form over $G/Z_\Theta$, the maximal geodesics are the curves
\begin{equation*}
	\gamma (t) = \pi_G (g \mathrm{exp} (tX)), \quad g\in G, \quad X\in \mathfrak{m}_G
\end{equation*}
where $\pi_G : G \rightarrow G/Z_\Theta $ is the canonical projection defined by $g\in G \mapsto g\cdot o$, with $o = 1\cdot Z_\Theta $. The maximal geodesics in $G/P_\Theta = U/U_\Theta$ are the curves 
\begin{equation*}
	\gamma _U (t) = \pi_U (u\, \mathrm{exp} (tA)), \quad u\in U, \quad A\in \mathfrak{m}\cong T_{b_\Theta}\mathbb{F}_\Theta
\end{equation*}
where $\pi_U :U \rightarrow U/U_\Theta$ is the canonical projection defined by $u\in U \mapsto u\cdot b_\Theta$.

Furthermore, we can define another projection
\begin{equation*}
	\begin{array}{rccccl}
		\pi :& G/Z_\Theta & \rightarrow & G/P_\Theta& \rightarrow & U/U_\Theta\\
		&gZ_\Theta & \mapsto & gP_\Theta &\mapsto & uU_\Theta
	\end{array}
\end{equation*}

 The following question arises: \emph{is the projection of a geodesic in $G/Z_\Theta $ a geodesic in $U/U_\Theta$?} Here, we analyze a case where the answer is yes. We will focus on studying geodesics that pass through the origin. 

From the decomposition in \eqref{decompoenfib}, if $[X_\mathfrak{m},X_f]=0$, then \[\mathrm{exp}(tX) =\mathrm{exp} (tX_m) \, \mathrm{exp}(tX_f)\] and
\begin{equation*}
	\pi (\mathrm{exp}(tX)\cdot o) = \mathrm{exp}(tX_\mathfrak{m}\cdot b_\Theta),
\end{equation*} since $X_m \in \mathfrak{m}$ and $X_f \in \mathfrak{n}_\Theta^+$.

When $X_f=0$, we get that the geodesic $\gamma (t)= g\exp(tX_\mathfrak{m} \cdot o)$ is a horizontal curve in the vector bundle $G/Z_\Theta \rightarrow U/U_\Theta$, and the projection in $U/U_\Theta$ is the geodesic curve $\gamma_U (t)=u \exp(tX_\mathfrak{m}\cdot b_\Theta)$, where $\pi_{GU}(g)= u$ is the projection of $G$ in $U$. We also have the case when $X_\mathfrak{m}=0$: the geodesic is in the fiber $\mathfrak{n}_\Theta^+$, and the projection of this curve on $U/U_\Theta$ is $\gamma_U (t) = b_\Theta$.

From the realization $T^*\mathbb{F}_\Theta= G/Z_\Theta$, the elements in $g\cdot o \in G/Z_\Theta$ can be viewed as $uv\cdot H_\Theta = u\cdot (H_\Theta + X) \in T^*\mathbb{F}_\Theta$, with $u\in U, v\in N_\Theta^+$ such that $g=uv$ and $v\cdot H_\Theta = H_\Theta +X$. Thus, the projection is given by $u\cdot (H_\Theta +X) \mapsto u\cdot H_\Theta$. Hence, for $\{ u_t \} \subset U$ and $\{ H_\Theta + X_t =e^{W_t} \cdot H_\Theta  \} \subset N_\Theta^+ \cdot H_\Theta $, we get that
\begin{equation*}
	\gamma (t) = e^{tZ} \cdot H_\Theta = u_t e^{W_t} \cdot H_\Theta = u_t \cdot (H_\Theta + X_t) .
\end{equation*}
Then, the projection of the geodesic is $\pi (\gamma (t)) = u_t \cdot H_\Theta$.  

When $Z\in \mathfrak{m}$, the geodesic $\gamma (t)= g e^{tZ} \cdot o$ is a horizontal curve in the vector bundle for $g\in G$. Its projection on the flag is the geodesic $\gamma_U (t) =  u e^{tZ} \cdot b_\Theta$, with $\pi_{GU}(g) =u\in U$. Moreover, when $Z\in \mathfrak{n}_\Theta^+$, then
\[\begin{array}{ccc}
	\gamma (t)= g\exp(tZ) \cdot o & \xrightarrow{\pi} & \gamma_U (t)=u\cdot b_\Theta.
\end{array}\] Such a geodesic lies in the fiber $u\cdot (H_\Theta + \mathfrak{n}_\Theta^+)$ of the cotangent bundle, hence being vertical.

\subsection{Orthogonal symmetric Lie algebras and conditions to nonnegative curvature on bundles}
\label{sec:curvature}
Let $(\mathfrak{g},\mathfrak{h},\sigma)$ be a symmetric Lie algebra. Consider the Lie algebra $\mathrm{ad}_\mathfrak{g}(\mathfrak{h})$ of linear endomorphisms of $\mathfrak{g}$ consisting of $\mathrm{ad}X$ where $X\in \mathfrak{h}$. If the connected Lie group of linear transformations of $\mathfrak{g}$ generated by $\mathrm{ad}_\mathfrak{g}(\mathfrak{h})$ is compact, then $(\mathfrak{g},\mathfrak{h},\sigma)$ is called an \emph{orthogonal symmetric Lie algebra}.

Let $(G, H,\sigma)$ be a symmetric space such that $H$ has a finite number of connected components, and suppose that $(\mathfrak{g},\mathfrak{h},\sigma)$ is its symmetric Lie algebra. In that case, it can be checked that $\mathrm{Ad}_G(\lie h)$ is compact if, and only if, $(\mathfrak{g},\mathfrak{h},\sigma)$ is an orthogonal symmetric Lie algebra. Furthermore, it holds that
\begin{proposition}
	Let $(\mathfrak{g},\mathfrak{h},\sigma)$ be an orthogonal symmetric Lie algebra with $\mathfrak{g}$ simple. Let $\mathfrak{g}=\mathfrak{h}+\mathfrak{m}$ be the canonical decomposition. Then,\begin{enumerate}
		\item $\mathrm{ad}\mathfrak{h}$ is irreducible on $\mathfrak{m}$.
		\item The Cartan--Killing form $\mathcal{K}$ of $\mathfrak{g}$ is (negative or positive) defined on $\mathfrak{m}$.
 	\end{enumerate}
\end{proposition}

A symmetric Lie algebra $(\mathfrak{g},\mathfrak{h},\sigma)$ is said to be \textit{effective} if $\mathfrak{h}$ contains no non-zero ideal of $\mathfrak{g}$.

An effective symmetric Lie algebra $(\mathfrak{g},\mathfrak{h},\sigma)$ is \emph{irreducible} if $\mathrm{Ad}([\mathfrak{m},\mathfrak{m}])$ is a irreducible representation on $\mathfrak{m}$. If $\mathfrak{g}$ is simple, then the symmetric Lie algebra is irreducible. In general, an orthogonal symmetric Lie algebra $(\mathfrak{g},\mathfrak{h},\sigma)$ with $\mathfrak{g}$ semisimple is said to be of \emph{compact type} or \emph{noncompact type} according as the Cartan--Killing form $\mathcal{K}$ of $\mathfrak{g}$ is negative-defined or positive definite on $\mathfrak{m}$. In what follows, we must consider Riemannian symmetric spaces, which, according to the former proposition, are related to orthogonal symmetric Lie algebras for semisimple lie algebras. Our main goal is to take advantage of the symmetric space realization of the cotangent bundle to understand the geometry of some vector bundles. More precisely, to understand the existence or not of metrics of nonnegative sectional curvature on some bundles; see, for instance, \cite{Strake1989, Chaves1992,Belegradek2001,10.1215/S0012-7094-03-11613-0}.

Since Cheeger--Gromoll's Soul Theorem: \emph{any complete
noncompact manifold $M$ with a metric of nonnegative sectional curvature is diffeomorphic to the normal bundle of some compact totally geodesic manifold} $\Sigma \subset M$, named as \emph{Soul} of $M$, it emerges the question, seen as a converse to this result: \emph{which vector
bundles over compact, nonnegatively curved base spaces can admit complete metrics
of nonnegative curvature?}

On the one hand, there are vector bundles that are known not to admit
nonnegative curvature, but, at least to the best of the authors' knowledge, in all such examples, the base space has an infinite fundamental group, see \cite{10.1215/S0012-7094-03-11613-0}.

On the other hand, trivial positive results include all homogeneous vector
bundles over homogeneous spaces: let $K{<\,}H{<\,}U$ be compact Lie groups and consider the submersion $\pi\co (U/K,\cal H)\to (U/H,b)$ defined by $\pi(gK)=gH$, where $\cal H$ encodes the horizontal distribution defining the submersion. Suppose that $(U/H,b)$ is normal homogeneous, where $b$ is the normal homogeneous Riemannian metric on $U/H$, and $g\cdot \cal H_p=\cal H_{gp}$ for all $g\in G$. Then, $U/K$ admits a Riemannian submersion metric $\tilde{\textsl g}$ of nonnegative sectional curvature with totally geodesic fibers.

Let $G$ be a complex simple Lie group, and let $U$ be its compact real form. Considering the previous discussion, we also take $K = \{e\}$ and let $H$ be the intersection $U\cap P_{\Theta}$ where $P_{\Theta}$ is a parabolic subgroup of $G$. Then $U/H$ is a flag manifold. According to Theorem \ref{thm:realizations}, we have that each pair $(U, H)$ there presented is a flag manifold that is also symmetric and whose cotangent bundle $T^*(U/H)$ is a symmetric space. Moreover, the submanifold $\cal S \subset T^*(U/H)$ is a symmetric space dual to $U/H$. We take the vector bundle
\[\bar \pi :\cal S\hookrightarrow \widetilde M := \left(U\times\cal S/\star \right)\rightarrow U/H\]
Moreover, observe that its fibers coincide with $\cal S$ while it defines a vector bundle over $U/H$. Here, $\widetilde M$ is the orbit space of the following $H$ action on $U\times \cal S$ 
\[(H\ni h) \star \left((u,v) \in U\times \cal S\right) := (uh,\mathrm{Ad}(h^{-1})v) \in U\times \cal S.\]
The projection $\bar \pi([u,v]) := \pi(u)$ where $\pi : U\rightarrow U/H$ is the orbit map $u\mapsto uH$.

Fixed the origin $o = eH \in U/H$ we denote $T_o(U/H) := \lie m$ where $\lie u = \lie m\oplus \lie{u}_{\Theta}$ corresponds to the appropriate Cartan decomposition of $\lie u$. Since $\cal S = \exp(\sqrt{-1}\lie m)$ one can define a Riemannian submersion metric on $\bar \pi$ in the following manner:
\[\ga(X+\sqrt{-1}Z,Y+\sqrt{-1}W) := -\cal K(X,Y) + \cal K(\sqrt{-1}Z,\sqrt{-1}W),~X,Y,Z,W \in \lie m\]
where $\cal K$ is the Cartan--Killing form of $\lie u$. Note that $X,Z \in \lie m$ while $\sqrt{-1}Z, \sqrt{-1}W\in \sqrt{-1}\lie m \cong T_o(U^*/H)$, where $U^*/H \cong S$ is the dual symmetric space to $U/H$. We also observe that any sphere bundle obtained out $\bar \pi$ admits a submersion metric with nonnegative curvature since it constitutes a homogeneous submersion: recall that $\cal S$ can be injectively projected in $U/H$ (Theorem \ref{thm:aS}). However, let us show that $\ga$ does not have nonnegative sectional curvature.

Observe that $U/H$ contains the topology of the bundle $\bar \pi$ since its fibers are contractible. We then consider $U/H$ as a \emph{Soul} to $\bar \pi$. We make use of the following result of Tapp--Strake--Walschap (see Theorem A and Corollary 2.5 in \cite{10.1215/S0012-7094-03-11613-0}): in order to $\cal S \hookrightarrow (M,\ga) \rightarrow (U/H,-\cal K)$ to have non-negative sectional curvature it must holds that
\[R_{\ga}(X,Y,\sqrt{-1}Z,\sqrt{-1}W)^2 \leq 4\mathrm{sec}_{U/H}(X,Y)\mathrm{sec}_{U^*/H}(\sqrt{-1}Z,\sqrt{-1}W),\]
where $R_{\ga}(X,Y,\sqrt{-1}Z,\sqrt{-1}W) := \ga(R_{\ga}(X,Y)\sqrt{-1}Z,\sqrt{-1}W)$ and $R_{\ga}$ is the Riemannian tensor associated to $\ga$. However, since $U/H$ and $U^*/H$ are symmetric spaces dual to each other, there exists $X, Y, Z, W \in \lie m$ in a way that the former inequality can not hold, finishing the proof Theorem \ref{thm:ziller}.

 We have then presented the first examples of vector bundles over simply connected manifolds that can not admit metrics of nonnegative sectional curvature, with the base manifold of nonnegative sectional curvature and finite fundamental group, but vector bundles for which the corresponding sphere bundles can admit metrics of nonnegative sectional curvature. We solve, in particular, Problem 9 in \cite{Ziller_fatnessrevisited}.

\section{Explicit descriptions of symmetric flag spaces}
\label{sec:proofmain}

In this section we find the automorphisms $\sigma$ of the symmetric spaces $(U,U_\Theta, \sigma)$, its dual symmetric space $(U^* ,(U^*)^\sigma)$ and $(G,G^\sigma)$. We search for an element $H_\Theta$ such that $g_\Theta := \exp (H_\Theta)$ and $H_\Theta$ has the same centralizer in $G$ and $U$, i.e.,
\begin{equation}	\label{introeq}
	Z_G(H_\Theta) = Z_G(g_\Theta) \qquad Z_U(H_\Theta)=Z_U(g_\Theta). 
\end{equation} 

The automorphism $\sigma$ is given by $C_{g_\Theta}$ for symmetric spaces and symmetric Lie algebras; we use the same notation $\sigma$ to the automorphism Ad$(g_\Theta)$.

	Assuming  that $p+q=l+1$ and $p\leq q$, we get
\[\begin{array}{|c|c|c|c|}
\hline
	G& \mathrm{Sl}(l+1,\mathbb{C}) &  \mathrm{Sp}(l,\mathbb{C}) & \mathrm{SO}(2l,\mathbb{C})	\\
	\hline
	
	\Theta & \Sigma \setminus \{ \alpha_p \} & \Sigma \setminus \{ \alpha_l \} &  \Sigma \setminus \{ \alpha_l \} \\
	\hline
	\mathbb{F}_\Theta & \mathrm{SU}(p+q)/\mathrm{S}(\mathrm{U}(p) \times \mathrm{U}(q)) &  \mathrm{Sp}(l)/\mathrm{U}(l) &\mathrm{SO}(2l)/\mathrm{U}(l) \\
	\hline
	T^* \mathbb{F}_\Theta & G/\mathrm{S}(\mathrm{Gl}(p,\mathbb{C}) \times \mathrm{Gl} (q,\mathbb{C})) & G/ \mathrm{Gl}(l,\mathbb{C}) & G / \mathrm{Gl}(l,\mathbb{C})\\
	\hline
\end{array}\]

The corresponding $H_\Theta$ satisfying \eqref{introeq} for $A_l, C_l $ and $D_l$ cases are

\begin{equation*}
	\dfrac{\sqrt{-1}\pi}{l+1}\left( \begin{array}{cc}
		qI_p&0\\
		0&-pI_q
	\end{array}\right) , \,\,\, \dfrac{\sqrt{-1}\pi}{2} \left( \begin{array}{cc}
		\mathrm{Id}_l&0\\ 0&-\mathrm{Id}_l
	\end{array} \right), \,\,\, \dfrac{\sqrt{-1}\pi}{2} \left( \begin{array}{cc}
		\mathrm{Id}_l&0\\0&-\mathrm{Id}_l \end{array} \right),
\end{equation*}
respectively.

\subsection{Case $A_l$}

Let $G$ be the Lie group $\mathrm{Sl}(n, \mathbb{C})$ and let $U$ be the compact Lie group $\mathrm{SU}(n)$, for $n=l+1$, with Lie algebras $\mathfrak{g}$ and $\mathfrak{u}$, respectively. The corresponding flag symmetric space is the set of Grassmannian of complex $p$-dimensional subspaces of $\mathbb{C}^{p+q}$ $$\mathbb{F}_{\Theta} = \mathrm{SU}(p+q)/\mathrm{S}(\mathrm{U}(p)\times \mathrm{U}(q)),\qquad n=p+q,\; p\leq q, $$ with 
\begin{gather*}
\Theta = \Sigma \setminus \{ \alpha_p \} = \\ \{ \alpha_1 = \lambda_1 - \lambda_2, \cdots ,  \alpha_{p-1}= \lambda_{p-1}-\lambda_{p} , \alpha_{p+1}= \lambda_{p+1}-\lambda_{p+2} , \cdots , \alpha_l = \lambda_l - \lambda_{l+1} \}    
\end{gather*} where each $\lambda_i$ is given by
\begin{equation*}
	\lambda_i : \mathrm{diag} \{  a_1 , \cdots , a_{l+1} \} \mapsto a_i.
\end{equation*}

	The compact real form of $\mathfrak{g}$ is $\mathfrak{u}= \mathfrak{su}(n)$ and an element in $\mathfrak{u}$ that defines the flag $\mathbb{F}_\Theta$ is the matrix 
\begin{equation}	\label{HThetaA}
	H_\Theta:=\dfrac{i\pi}{n}\left( \begin{array}{cc}
		qI_p&0\\
		0&-pI_q
	\end{array}\right) \in \mathfrak{su}(n)\subset \mathfrak{sl}(n,\mathbb{C}).
\end{equation} 
It belongs to the Cartan subalgebra of $\mathfrak{su}(n)$ and its exponential is the matrix \begin{equation}
	g_\Theta:=\exp(H_\Theta)=\left( \begin{array}{cc}
		e^{i\frac{\pi q}{n}}I_p&0\\
		0&e^{-i\frac{p\pi}{n} }I_q
	\end{array} \right) .
\end{equation}
Since $p+q=n$, then $e^{i\frac{\pi}{n}(p+q)}=-1$. So, we have
$e^{i\pi \frac{q}{n}}\cdot e^{i\pi \frac{p}{n}} = -1$. Therefore, $e^{i\pi\frac{q}{n}}=-e^{-i\pi \frac{p}{n}}$
and 
$g_\Theta=e^{i\pi \frac{q}{n}} I_{p,q} \in \mathrm{SU}(n)$, where $I_{p,q} := 
\begin{matrix}
	\left( 	\begin{array}{cc}
		I_p&0\\
		0&-I_q
	\end{array} \right) 
\end{matrix} $.

\

Furthermore, $(g_\Theta)^2 = (e^{i\pi \frac{q}{n}})^2 \cdot I_{p,q}^2 = e^{2i \pi \frac{q}{n}}\cdot I_n$. This verifies that $(g_\Theta^2)^n = e^{2\pi q i}\cdot I_n = I_n$, i.e., $(g_\Theta)^2\in Z(\mathrm{SU}(n))$, since $g_\Theta^2$ is a $n$-th root of unit.

We shall use the same notation $\sigma$ for the symmetric automorphism for Lie groups and algebras. Seeing that $I_{p,q}^{-1}=I_{p,q}$, the inner automorphism 
$$	\begin{array}{rccl}
	\sigma :& \mathrm{SU}(n) & \rightarrow & \mathrm{SU}(n)\\
	&h & \mapsto & C_{g_\Theta}(h)=g_\Theta hg_{\Theta}^{-1} = e^{i\pi \frac{q}{n}}I_{p,q}he^{-i\pi \frac{q}{n}}I_{p,q}^{-1}=I_{p,q}hI_{p,q}
 \end{array}$$ is an involution.

 The fixed point set of $C_{g_\Theta}$ in $\mathrm{SU}(n)$ is computed as follows:
\begin{align*}
	\mathrm{fix} (C_{g_\Theta}) =&	\{h\in \mathrm{SU}(n): I_{p,q}hI_{p,q}=h \} 	\nonumber			\\
	=&\{ h\in \mathrm{SU}(n) : I_{p,q}=h^{-1}I_{p,q}h \} = Z_{\mathrm{SU}(n)} (I_{p,q}) \\
	=&\mathrm{S}(\mathrm{U}(p)\times \mathrm{U}(q)) .
\end{align*}

Hence, the flag manifold $	\mathrm{SU}(n)/\mathrm{S}(\mathrm{U}(p)\times \mathrm{U}(q))$ is a symmetric space. The tangent space to the symmetric flag at the origin is isomorphic to
\begin{equation} \label{primerm}
	\mathfrak{m}= \left\lbrace \left(  \begin{array}{cc}
		0&Z\\-\bar{Z}^T&0
	\end{array} \right) : Z_{p\times q} \; \mathrm{complex} \, \mathrm{matrix} \right\rbrace 
\end{equation}
and the symmetric pair is $(\mathfrak{su}(n), \mathfrak{u}_\Theta)$, where the subalgebra  $\mathfrak{u}_\Theta$ consists of all elements in $\mathfrak{su}(n)$ which are fixed by $C_{g_\Theta}$. This is also the Lie algebra of $\mathrm{S}(\mathrm{U}(p)\times \mathrm{U}(q))$ 
\begin{equation} \label{tsl}
	\mathfrak{u}_\Theta=
	\left\lbrace  \left( \begin{array}{cc}
		A_{p\times p}&0\\ 0 & B_{q\times q}
	\end{array} \right): \mathrm{tr}(A)+ \mathrm{tr}(B)=0, A\in \mathfrak{u}(p), \; B\in \mathfrak{u}(q) \right\rbrace .
\end{equation}

To study the symmetry of the cotangent bundle of the symmetric flag space $\mathbb{F}_\Theta$, we study the centralizer of $H_\Theta$ in the Lie group $G$,
\begin{align}
	Z_{\mathrm{Sl}(n,\mathbb{C})}(H_\Theta)=& \{ h\in \mathrm{Sl}(n,\mathbb{C}) : \mathrm{Ad}(h)H_\Theta = H_\Theta \} \nonumber  \\
	=& \left\{ h = \left( \begin{array}{cc}
		a_{p\times p} & b\\ c&d_{q\times q}
	\end{array} \right)  \in \mathrm{Sl}(n, \mathbb{C}) : h \left( \begin{array}{cc}
		qI_p & 0\\ 0&-pI_q
	\end{array} \right) = \left( \begin{array}{cc}
		qI_p & 0\\ 0&-pI_q
	\end{array} \right) h \right\} \nonumber \\
	=& \left\{ h\in \mathrm{Sl}(n,\mathbb{C}): \left( \begin{array}{cc}
		qa & -pb\\ qc&-pd
	\end{array} \right) = \left( \begin{array}{cc}
		qa & qb\\ -pc&-pd
	\end{array} \right) \right\} \nonumber \\
	=& \left\{ \left( \begin{array}{cc}
		a & 0\\ 0&d
	\end{array} \right) \in \mathrm{Sl}(n,\mathbb{C}) : a_{p\times p}, d_{q\times q} \right\} \nonumber \\
	=& \mathrm{S}(\mathrm{Gl}(p,\mathbb{C})\times \mathrm{Gl}(q,\mathbb{C}))
\end{align} and note that the centralizer of the matrix $I_{p,q}$ in the group $\mathrm{Sl}(n,\mathbb{C})$ is

\begin{align*}
	Z_{\mathrm{Sl}(n,\mathbb{C})}(I_{p,q}) =& \left\{ h = \left( \begin{array}{cc}
		a_{p\times p} & b\\ c&d_{q\times q}
	\end{array} \right)  \in \mathrm{Sl}(n, \mathbb{C}) : h  \left( \begin{array}{cc}
		I_p & 0\\ 0&-I_q
	\end{array} \right) = \left( \begin{array}{cc}
		I_p & 0\\ 0&-I_q
	\end{array} \right) h \right\} \nonumber \\
	=& \left\{ h\in \mathrm{Sl}(n,\mathbb{C}): \left( \begin{array}{cc}
		a & -b\\ c&-d
	\end{array} \right) = \left( \begin{array}{cc}
		a & b\\ -c&-d
	\end{array} \right) \right\} \nonumber \\
	=& \left\{ \left( \begin{array}{cc}
		a & 0\\ 0&d
	\end{array} \right) \in \mathrm{Sl}(n,\mathbb{C}) : a_{p\times p}, d_{q\times q} \right\} \nonumber \\
	=& \mathrm{S}(\mathrm{Gl}(p,\mathbb{C})\times \mathrm{Gl}(q,\mathbb{C})) \\
	=& Z_{\mathrm{Sl}(n, \mathbb{C})}(H_\Theta)
\end{align*}

We can extend the inner automorphism $\sigma = C_{g_\Theta}$ to the complex Lie group $\mathrm{Sl}(n,\mathbb{C})$:
\begin{equation}	\label{autoA}
	\begin{array}{rccl}
		C_{g_\Theta}:&\mathrm{Sl}(n,\mathbb{C}) & \rightarrow & \mathrm{Sl}(n,\mathbb{C}) \\
		&h & \mapsto & ghg^{-1} = I_{p,q}hI_{p,q}.
	\end{array}
\end{equation} Since we saw that it is an involution and its fixed point set is $Z_{\mathrm{Sl}(n,\mathbb{C})}(I_{p,q}) = Z_{\mathrm{Sl}(n,\mathbb{C})}(H_\Theta) = \mathrm{S}(\mathrm{Gl}(p,\mathbb{C})\times \mathrm{Gl}(q,\mathbb{C}))$, one gets that the cotangent bundle of $\mathbb{F}_\Theta$, as homogeneous space, \[\mathrm{Sl}(n,\mathbb{C})/ \mathrm{S}(\mathrm{Gl}(p,\mathbb{C})\times \mathrm{Gl}(q,\mathbb{C}))\] is a symmetric space associated to the symmetric pair $(\mathfrak{sl}(n,\mathbb{C}), \mathfrak{sl}(p,\mathbb{C})\oplus\mathfrak{sl}(q,\mathbb{C})\oplus \mathbb{C}^*)$. In fact, let $A\in \mathfrak{s}( \mathfrak{gl}(p,\mathbb{C}) \times \mathfrak{gl}(q,\mathbb{C}))$, then $$A=\left( \begin{array}{cc}
	D_{p\times p}&0 \\ 0&B_{q\times q} \end{array} \right) .$$ 
\begin{itemize}
	\item[(i)] 	If $A\in \mathfrak{sl}(p,\mathbb{C})\oplus\mathfrak{sl}(q,\mathbb{C})$, then $$A= \left[ A+\left( \begin{array}{cc}
		-I_{p}&0 \\ 0&\frac{p}{q} I_{q} \end{array} \right) \right]   + \left( \begin{array}{cc}
		I_p&0\\0&-\frac{p}{q}I_{q}
	\end{array} \right) \mapsto A' + \frac{p}{q}.$$
	\item[(ii)]  If $A\notin \mathfrak{sl}(p,\mathbb{C})\oplus\mathfrak{sl}(q,\mathbb{C})$, then $$A= \left[ A+\mathrm{tr}(B) \left( \begin{array}{cc}
		I_{p}&0 \\ 0&- I_{q} \end{array} \right) \right]   + \mathrm{tr}(B) \left( \begin{array}{cc}
		-I_p&0\\0&I_{q}
	\end{array} \right) \mapsto A' + \mathrm{tr}(B).$$
\end{itemize}
In both of cases, $A'$ belongs to $\mathfrak{sl}(p,\mathbb{C})\oplus\mathfrak{sl}(q,\mathbb{C})\oplus \mathbb{C}^*$. Hence $\mathfrak{s}( \mathfrak{gl}(p,\mathbb{C}) \times \mathfrak{gl}(q,\mathbb{C})) \subset \mathfrak{sl}(p,\mathbb{C})\oplus\mathfrak{sl}(q,\mathbb{C})\oplus \mathbb{C}^* $. The other containment is straightforward. Thus, this symmetric Lie algebra is in Table \ref{Table1}. Furthermore, the tangent space of that symmetric space at the origin is isomorphic to 
\begin{align}
	\mathfrak{m}_G =& 	\mathfrak{m} + \sqrt{-1} \mathfrak{m} \nonumber \\
 \cong & \mathfrak{n}_\Theta^- + \mathfrak{n}_\Theta^+ \nonumber \\
	=& \left\lbrace \left( \begin{array}{cc}
		0&A\\B&0
	\end{array} \right) : A_{p\times q} , \; B_{q\times p} \; \mathrm{complex}\; \mathrm{matrices} \right\rbrace \label{soma}
\end{align}
where the elements in $\mathfrak{n}_\Theta^-$ are the matrices in \ref{soma}  with $A=0$, $\mathfrak{n}_\Theta^+$ is the set \eqref{soma} with $B=0$ and $\mathfrak{m}$ is the set in \eqref{primerm}.

Directly applying Theorem \ref{orbitafibradoSM}, we get that $$ G\cdot H_\Theta = G/Z_\Theta = T^* (U/U\cap Z_\Theta ) = T^*(U/U_\Theta).$$
Then,
\[\mathrm{Sl}(n,\mathbb{C})/ \mathrm{S}(\mathrm{Gl}(p,\mathbb{C})\times \mathrm{Gl}(q,\mathbb{C})) = T^* (\mathrm{SU}(n)/\mathrm{S}(\mathrm{U}(p)\times \mathrm{U}(q))).\]

Thus, the cotangent bundle of the flag symmetric manifold $\mathrm{SU}(n)/\mathrm{S}(\mathrm{U}(p)\times \mathrm{U}(q))$ is also a symmetric space with the same involutive automorphism $\sigma = C_{g_\Theta}$.

Note that the inner automorphism can be lifted to the Lie algebra $\mathfrak{g}=\mathfrak{sl}(n,\mathbb{C})$, obtaining the function
\begin{equation}	\label{involsl}
	\begin{array}{rccl}
		\sigma =\mathrm{Ad}(g_\Theta ):&\mathfrak{sl}(n,\mathbb{C}) & \rightarrow & \mathfrak{sl} (n,\mathbb{C}) \\
		&Z & \mapsto & g_\Theta Zg_\Theta^{-1} = I_{p,q}ZI_{p,q} .
	\end{array}
\end{equation}

Denote by $\mathfrak{g}_0 $ the non compact real form $\mathfrak{su}(p,q)$ of $\mathfrak{g}$, which elements are the $n\times n$ complex matrices satisfying
\begin{equation}
	I_{p,q} X+\bar{X}^T I_{p,q} =0, \qquad I_{p,q} = \left( \begin{array}{cc}
		1_p&0\\0&-1_q
	\end{array} \right) . 
\end{equation} This Lie algebra is  
\begin{equation}
	\mathfrak{su}(p,q)= \left\{ \left( \begin{array}{cc}
		A_{p\times p}& B_{p\times q} \\ \bar{B}_{q\times p}^T & C_{q\times q}
	\end{array}  \right) : A \in \mathfrak{u}(p) , \; C\in \mathfrak{u}(q) ,\; \mathrm{tr}(A+C)=0 \right\},
\end{equation}and has the Cartan decomposition $\mathfrak{su}(p,q) = \mathfrak{k} + \mathfrak{s}$ with $\mathfrak{k} = \mathfrak{u}_\Theta$, as in \eqref{tsl}, and 
\begin{equation}
	\mathfrak{s} = \left\{ \left( \begin{array}{cc}
		0&B\\\bar{B}^T&0
	\end{array} \right) : B_{p\times q} \; \mathrm{complex}\; \mathrm{matrix} \right\}.
\end{equation}

	The involution Ad(exp$(H_\Theta)$) in $\mathfrak{sl}(n,\mathbb{C})$, 
\begin{equation}
	\left( \begin{array}{cc}
		A&B\\C&D
	\end{array}
	\right) \mapsto \left( \begin{array}{cc}
		A&-B\\-C&D
	\end{array} \right)
\end{equation}
restricted to $\mathfrak{su}(p,q)$ is also an inner automorphism since $H_\Theta \in \mathfrak{u}_\Theta = \mathfrak{k}$. It also coincides with the Cartan involution $Z \in \mathfrak{su}(p,q) \mapsto -\bar{Z}^T \in \mathfrak{su}(p,q)$. 

From the definition of $\mathfrak{m}$ in \eqref{primerm}, we have that the vector space $\sqrt{-1}\mathfrak{m}$ coincides with $ \mathfrak{s}$ and therefore, the dual symmetric algebra of $(\mathfrak{su}(n), \mathfrak{k}, \sigma )$ is
\begin{equation}	\label{DSA}
	 \mathfrak{su}(n)^* = \mathfrak{su}(p,q)= \mathfrak{k} + \mathfrak{s}.
\end{equation}

In this way, as studied in Section \ref{DSS}, the dual symmetric space\linebreak $\mathrm{SU}(p,q)/\mathrm{S}(\mathrm{U}(p)\times \mathrm{U}(q))$ of $\mathrm{SU}(p+q)/\mathrm{S}(\mathrm{U}(p)\times \mathrm{U}(q))$ is a submanifold of the adjoint orbit $\mathrm{Sl}(n,\mathbb{C})\cdot H_\Theta$ (this $H_\Theta$ is from the $A_l$ case). 

The dual symmetric space $\mathrm{SU}(p,q)/\mathrm{S}(\mathrm{U}(p)\times \mathrm{U}(q))$ is diffeomorphic to $\mathcal{S} = \exp \mathfrak{s} \cdot H_\Theta$. To prove Proposition \ref{lemfibinterS}, let us calculate the intersection $\mathcal{S} \cap (H_\Theta + \mathfrak{n}_\Theta^+)$. Initially, note that
\begin{equation}	
	H_\Theta \in	(H_\Theta + \mathfrak{n}_\Theta^+ ) \cap \mathrm{Ad}(S) \cdot H_\Theta \subset \mathfrak{su}(p,q) \cap (H_\Theta + \mathfrak{n}_\Theta^+),
\end{equation}
where $S$ is the exponential of the vector space $\mathfrak{s}$.
To calculate the intersection $\mathfrak{su}(p,q) \cap (H_\Theta + \mathfrak{n}_\Theta ^+)$, we describe the elements of the fiber passing through the origin $H_\Theta$:
\begin{equation} \left( 
	\begin{array}{cc}
		\dfrac{	\sqrt{-1}\pi}{n}q I_p&Z_{p\times q}\\ 0 &-\dfrac{\sqrt{-1}\pi}{n}pI_q
	\end{array} \right) ,
\end{equation}
with $Z\in \mathfrak{n}_\Theta^+$. These matrices belong to $\mathfrak{su}(p,q)$ if, and only if, they satisfy the expression
\begin{equation*}
	\left( \begin{array}{cc}
		I_p &0\\0&-I_q
	\end{array} \right) \left( \begin{array}{cc}
		\frac{\sqrt{-1}\pi}{n}qI_p & Z\\0&-\frac{\sqrt{-1}\pi}{n}pI_q
	\end{array} \right) + \left( \begin{array}{cc}
		- \frac{\sqrt{-1}\pi}{n}qI_p &0\\ \bar{Z}^T & \frac{\sqrt{-1}\pi}{n}qI_p
	\end{array} \right)  \left( \begin{array}{cc}
		I_p&0\\0&-I_q
	\end{array} \right) = \mathbf{0}
\end{equation*}
which is equivalent to 
\begin{equation}
	\left( \begin{array}{cc}
		0&Z\\\bar{Z}^T &0
	\end{array} \right)  = \mathbf{0} \Longleftrightarrow Z= \mathbf{0}.
\end{equation}

It means that the intersection $\mathfrak{su}(p,q) \cap (H_\Theta + \mathfrak{n}_\Theta^+) \subset \{ H_\Theta \}$. Once from \eqref{Ainter}, we conclude that
\begin{equation}	
	\mathfrak{su}(p,q) \cap (H_\Theta + \mathfrak{n}_\Theta^+) = \{ H_\Theta \} .
\end{equation}
it holds that
\begin{equation}
	H_\Theta \in (H_\Theta + \mathfrak{n}_\Theta^+) \cap \mathcal{S} \subset (H_\Theta + \mathfrak{n}_\Theta^+) \cap \mathfrak{su}(p,q)^*(l) = \{ H_\Theta \} .
\end{equation}
Hence, 
\begin{equation}	\label{interfibA}
	(H_\Theta + \mathfrak{n}_\Theta^+) \cap \mathcal{S} = \{ H_\Theta \} .
\end{equation}

However, the fiber $-H_\Theta + \mathfrak{n}_\Theta^-$ is
\begin{equation}
	\left( 
	\begin{array}{cc}
		\dfrac{ -	\sqrt{-1}\pi}{n}q I_p&0 \\Z_{p\times q} &\dfrac{\sqrt{-1}\pi}{n}pI_q
	\end{array} \right) 
\end{equation}
while the element $H_\Theta$ is the linear combination \begin{gather*}
    H_\Theta = 2\sqrt{-1}\pi q [H_{\alpha_1} + 2 H_{\alpha_2} + \cdots+ pH_{\alpha_p}] \\+ 2\sqrt{-1} \pi p [ (q-1)H_{\alpha_{p+1}} + (q-2)H_{\alpha_{p+2}} + \cdots + 2H_{\alpha _{p+q-2}} + H_{\alpha_l}].
\end{gather*} 
We also have, $\mathcal{K}(\sqrt{-1} H_{\alpha_p} , H_\Theta) =-\pi.$

\subsection{Case $C_l$}\label{chpt4}

	Let $G$ be the Lie group $\mathrm{Sp}(l, \mathbb{C})$ and $U$ be the compact Lie group $\mathrm{Sp}(l) = \mathrm{Sp}(l, \mathbb{C})\cap \mathrm{SU}(2l)$, with Lie algebras $\mathfrak{g}$ and $\mathfrak{u}$, respectively. The corresponding flag symmetric manifold is the space of complex structures on  $\mathbb {H}^{l}$ compatible with its standard inner product, that is,
\begin{equation*}
	\mathbb{F}_\Theta =	\mathrm{Sp}(l)/\mathrm{U}(l)
\end{equation*}
with $\Theta = \{ \lambda_1 - \lambda_2 , \cdots , \lambda_{l-1} - \lambda_l \}$ and each $\lambda_i$ given by
\begin{equation}	\label{lambai}
	\lambda_i : \mathrm{diag}\{ a_1 , \cdots , a_l \} \mapsto a_i.
\end{equation}
The Lie algebra $\mathfrak{u}:=\mathfrak{sp}(l) =\mathfrak{sl}(l,\mathbb{C})\cap \mathfrak{su}(2l)$ is the set
\begin{equation}
	\mathfrak{sp}(l)= \left\lbrace \left(  \begin{array}{cc}
		a_{l\times l}&b\\-\bar{b}&\bar{a}
	\end{array} \right) \mid  a\in \mathfrak{u}(l), b: \, \mathrm{symmetric} \, \mathrm{complex} \, \mathrm{matrix} \right\rbrace 
\end{equation}
which is a compact real form of the Lie algebra of $G$, denoted by $\mathfrak{g}:=\mathfrak{sp}(l,\mathbb{C}) = \{ A\in M(2l,\mathbb{C}): AJ+JA^T=0 \}$, where $J=\left( \begin{array}{cc}
	0&-Id_l\\Id_l &0
\end{array} \right) $ with $J^2 = -Id_{2l}$. The elements of $\mathfrak{g}$ are $2l \times 2l $ complex matrices which can be represented in blocks as
\begin{equation}
	\left( \begin{array}{cc}
		a_{l\times l}&b\\c&-a^T
	\end{array} \right) ,\qquad b,c: \; \mathrm{symmetric}.
\end{equation}

The matrix above is quaternionic since $a,b,c$ are complex matrices. Furthermore, the Lie group $G$ is the set
\begin{eqnarray}
	\mathrm{Sp}(l,\mathbb{C})=& \left\lbrace  M\in M_{2l\times 2l }(\mathbb{C}): M^T \Omega M = \Omega \right\rbrace ,\quad \Omega =\left( \begin{array}{cc}
		0&1_l\\-1_l&0
	\end{array} \right)  \label{omega} \\
	=& \left\lbrace M= \left( \begin{array}{cc} A&B\\C&D \end{array} \right) : A,B,C,D \in M_{l\times l}(\mathbb{C}), \mathrm{det}(M)=1 \right\rbrace 
\end{eqnarray}
with $-C^TA+A^TC=0$, $-C^TB+A^TD=I_l$, $-D^TB+B^TD=0$.

An element in the Cartan subalgebra $\sqrt{-1}\mathfrak{h}_\mathbb{R}$ of the compact real form $\mathfrak{u}$ that defines the symmetric flag manifold is
\begin{equation}	\label{thetac}
	H_\Theta := \dfrac{\pi}{2} \left( \begin{array}{cc}
		\sqrt{-1}Id_l&0\\ 0&-\sqrt{-1}Id_l
	\end{array} \right) \in \mathfrak{sp}(l)\subset \mathfrak{sp}(l,\mathbb{C})
\end{equation} and its exponential is the matrix
\begin{equation}
	g_\Theta :=e^{H_\Theta } =  \left( \begin{array}{cc}
		\sqrt{-1}Id_l&0\\ 0&-\sqrt{-1}Id_l
	\end{array} \right) 
\end{equation}

Since $H_\Theta=\frac{\pi}{2}g_\Theta$ and the centralizer $Z_\Theta = Z_G (g_\Theta)$ of $H_\Theta$ in $G$ satisfies
\begin{equation}
	g_\Theta^2 = -Id_{2l} \in Z(U) \subset Z(G),
\end{equation} 
and the inner automorphism of $\mathrm{Sp}(l, \mathbb{C})$ is given by 
\begin{equation} \label{invosp}
 \sigma = C_{g_\Theta}=g_\Theta hg_\Theta^{-1}.  
\end{equation}
Its fixed point set is:
\begin{align}
	Z_G(g_\Theta)=Z_\Theta	=& \left\lbrace   \left( \begin{array}{cc}
		a&0\\0&d
	\end{array} \right) \in \mathrm{Sp}(l,\mathbb{C})  \right\rbrace  \nonumber 	
\end{align}
We also have
\begin{eqnarray}
	Z_\Theta=& \left\lbrace \left( \begin{array}{cc}
		a&0\\0&(a^T)^{-1} 
	\end{array} \right) \in \mathrm{Sp}(l,\mathbb{C}) \right\rbrace  \cong\mathrm{Gl}(l,\mathbb{C}). \nonumber
\end{eqnarray}
Thus, 
\begin{equation}
	G/Z_G(g_\Theta) = \mathrm{Sp}(l,\mathbb{C})/ \mathrm{Gl}(l,\mathbb{C})
\end{equation} is a symmetric space with symmetric pair $(\mathfrak{sp}(l,\mathbb{C}),\mathfrak{sl}(l,\mathbb{C})\oplus \mathbb{C}^*)$.  In fact, let $A\in \mathfrak{gl}(n,\mathbb{C})$. 
\begin{itemize} \label{HIJOHIJOHIJO}
	\item[(i)] 	If $A\in \mathfrak{sl}(n,\mathbb{C})$, then $$A= \left[ A+\left( \begin{array}{cc}
		-1&0 \\ 0&\frac{1}{n-1} I_{n-1} \end{array} \right) \right]   + \left( \begin{array}{cc}
		1&0\\0&-\frac{1}{n-1}I_{n-1}
	\end{array} \right) \mapsto A' + \frac{1}{1-n} \in \mathfrak{sl}(n,\mathbb{C}) \oplus \mathbb{C}^*.$$
	\item[(ii)]  If $A\notin \mathfrak{sl}(n,\mathbb{C})$, then $A= \left[  A-\frac{tr(A)}{n}I_n \right]  + \frac{tr(A)}{n}I_n \mapsto \left[  A-\frac{tr(A)}{n}I_n \right] + \frac{tr(A)}{n} \in \mathfrak{sl}(n,\mathbb{C}) \oplus \mathbb{C}^*$. 
\end{itemize}
Hence $ \mathfrak{gl}(n,\mathbb{C})) \subset \mathfrak{sl}(n,\mathbb{C})\oplus \mathbb{C}^* $. The other containment is straightforward. Hence, this symmetric Lie algebra is in Table \ref{Table2}. Once more, Theorem \ref{orbitafibradoSM} implies that the cotangent bundle
\begin{equation}
	G/Z_\Theta  = G\cdot H_\Theta = T^*(U/U_\Theta)
\end{equation}is a symmetric space. Since $U_\Theta = Z_\Theta \cap \mathrm{Sp}(l) $ and using the fact that $\mathrm{Sp}(l) = \mathrm{Sp}(l,\mathbb{C})\cap \mathrm{SU}(2l) $, we have that
\begin{eqnarray}
	U_\Theta =& \left\lbrace \left( \begin{array}{cc}
		a&0\\0& (a^T)^{-1} 
	\end{array} \right) \in \mathrm{SU}(2l) : a_{l\times l} \; \mathrm{complex} \right\rbrace =\mathrm{U}(l)   \nonumber 
\end{eqnarray}

Hence, the cotangent bundle of $\mathrm{Sp}(l)/\mathrm{U}(l)$:
\begin{equation}
	T^*(\mathrm{Sp}(l)/\mathrm{U}(l)) = \mathrm{Sp}(l,\mathbb{C})/\mathrm{Gl}(l,\mathbb{C})
\end{equation}
is a symmetric space.

The Lie algebra $\mathfrak{g}=\mathfrak{sp}(l,\mathbb{C})$ has the real normal form $\mathfrak{g}_0 = \mathfrak{sp}(l,\mathbb{R}) $
\begin{equation}
	\mathfrak{g}_0 = \left\lbrace  \left( \begin{array}{cc}
		a_{l\times l}&b\\c &-a^T
	\end{array} \right) \mid b,c \; \mathrm{symmetrics} \right\rbrace  
\end{equation} 
which has the Cartan decomposition $$\mathfrak{g}_0 = \mathfrak{k} + \mathfrak{s}$$ where
\begin{equation} \label{kc}
	\mathfrak{k}=\left\lbrace  \left( \begin{array}{cc}
		a_{l\times l}&-b\\b&a
	\end{array} \right):  a \in \mathfrak{o}(l), b \,\, \mathrm{ symmetric} \right\rbrace   \cong \mathfrak{u}(l)
\end{equation}
\begin{equation}	\label{sc}
	\mathfrak{s}= \left\lbrace  \left( \begin{array}{cc}
		a_{l\times l}&b\\b&-a
	\end{array} \right) : a,b \, \, \mathrm{ symmetrics} \right\rbrace .
\end{equation}

The lift of the automorphism in \eqref{invosp} to the Lie algebra restricted to $\mathfrak{g}_0$ is
\begin{equation}	
	\begin{array}{cccc}
		\sigma = \mathrm{Ad}(g_\Theta):& \mathfrak{sp} (l, \mathbb{R} )&\rightarrow & \mathfrak{sp}(l,\mathbb{R})\\
		&X&\mapsto & \mathrm{Ad}(g_\Theta)X =g_\Theta Xg_\Theta^{-1}.
	\end{array}
\end{equation}
It does not coincide with the Cartan involution $\Theta (X) = -X^T$. However, we have the symmetric pairs $(\mathfrak{sp}(l),\mathfrak{u}(l))$ and $(\mathfrak{sp}(l,\mathbb{C}), \mathfrak{sl}(l,\mathbb{C}) \oplus \mathbb{C} )$, since the Lie algebra of $Z_\Theta$ is $\mathfrak{gl}(l,\mathbb{C}) \cong \mathfrak{sl}(l,\mathbb{C}) \oplus \mathbb{C}$. 

The tangent bundle of the cotangent bundle of $\mathbb{F}_\Theta$, as a homogeneous space at the origin, is isomorphic to 
\begin{eqnarray}
	\mathfrak{m}_G =&	\mathfrak{n}_\Theta^- + \mathfrak{n}_\Theta^+\\
	=& \left\lbrace    \left( \begin{array}{cc}
		0&Y_{l\times l}\\X_{l \times l }&0 
	\end{array} \right)  : X,Y \, \mathrm{symmetric} \; \mathrm{complex}\; \mathrm{matrices} \right\rbrace 	\label{mgc}	\\
	=& \mathfrak{m} + \sqrt{-1}\mathfrak{m}	
\end{eqnarray} where $\mathfrak{n}_\Theta^\pm$ is the set of matrices in \eqref{mgc} with $X=0$ and  $Y=0$, respectively. The tangent bundle of the symmetric flag manifold at the origin is isomorphic to
\begin{eqnarray}
	\mathfrak{m} \cong & \sum_{\alpha \in \Pi^+ \setminus \left\langle \Theta \right\rangle^+ } \mathfrak{u}_\alpha\\
	=& \left\lbrace \left( \begin{array}{cc}
		0&B\\-\bar{B}&0
	\end{array} \right) : B_{l\times l} \, \mathrm{symmetric}\; \mathrm{complex} \right\rbrace  \label{segundom}
\end{eqnarray}
where the subspace $\mathfrak{u}_\alpha =$span$_\mathbb{R}\{ A_\alpha, Z_\alpha \}$  is generated by $A_\alpha = \left( \begin{array}{cc}
	0&E_{ij}+E_{ji}\\-E_{ij}-E{ji}&0
\end{array} \right) $ and $Z_\alpha = \left( \begin{array}{cc}
	0&\sqrt{-1}(E_{ij}+E_{ji})\\ \sqrt{-1}(E_{ij}+E_{ji}) &0
\end{array} \right) $ with $\alpha \in \Pi^+ \setminus \langle \Theta \rangle ^+ = \{ \lambda_{i} + \lambda_{j}, 1\leq i,j\leq l \}$, for $E_{i,j} = \delta_i^j Id$.\\

	The involution $\sigma$ for the symmetric Lie algebra $(\mathfrak{sp}(l,\mathbb{C}), \mathfrak{gl}(l,\mathbb{C}) ,\sigma))$ is the following automorphism, with $H_\Theta \in \mathfrak{sp}(l)$ as in \eqref{thetac},
\begin{equation}	\label{adjunta}
\mathrm{Ad}(e^{H_\Theta}):	\left(\begin{array}{cc}
		A_{l\times l}&B\\C&-A^T
	\end{array}\right) \in \mathfrak{sp}(l,\mathbb{C}) \mapsto \left( \begin{array}{cc}
		A&-B\\-C&-A^T
	\end{array} \right). 
\end{equation}

The involution above is also an automorphism of  $\mathfrak{sp}(l)$, and this symmetric Lie algebra has the canonical decomposition 
\begin{equation}
	\mathfrak{sp}(l) = 
	\mathfrak{u}_\Theta+ \mathfrak{m},
\end{equation}
with $\mathfrak{m}$ as in \eqref{segundom} and $\mathfrak{u}_\Theta$ being the subalgebra such that $\sigma (W) = W$, for all $W\in \mathfrak{u}_\Theta$

\begin{equation}
	\mathfrak{u}_\Theta = \left\{ \left( \begin{array}{cc}
		A&0\\0&\bar{A}
	\end{array} \right) : A\in \mathfrak{u}(l) \right\} \cong \mathfrak{u}(l).
\end{equation} 

Then, the dual symmetric Lie algebra of $(\mathfrak{sp}(l), \mathfrak{u}(l) ,\sigma )$ must have the canonical decomposition
\begin{equation}
	\mathfrak{sp}(l)^* = \mathfrak{u}_\Theta + \sqrt{-1}\mathfrak{m}. 
\end{equation}

The set $\sqrt{-1}\mathfrak{m}$ is the vector space satisfying $\sigma (W) = -W$, for all $W\in \sqrt{-1}\mathfrak{m}$, it is
\begin{equation}
	\left\{ \left( \begin{array}{cc}
		0&B_{l\times l}\\ \bar{B}&0
	\end{array} \right) : A \; \mathrm{symmetric}\; \mathrm{complex} \; \mathrm{matrix} \right\}
\end{equation}
and then
\begin{equation} \label{spl*}
	\mathfrak{sp}(l)^* = \left\{ \left( \begin{array}{cc}
		A&B \\ \bar{B}&\bar{A}
	\end{array} \right) : A\in \mathfrak{u}(l), B\; \mathrm{symmetric}\; \mathrm{complex} \; \mathrm{matrix}  \right\} .
\end{equation}

Let us define the following function
\begin{equation}	\label{Cayley}
	\left( \begin{array}{cc}
		A&B\\ \bar{B}&\bar{A}  
	\end{array} \right) \in \mathfrak{sp}(l)^* \mapsto \left( \begin{array}{cc}
		\mathrm{Re}(A)+\mathrm{Re}(B)& \mathrm{Im}(B)-\mathrm{Im}(A)\\ \mathrm{Im}(A)+\mathrm{Im}(B) & \mathrm{Re}(A)-\mathrm{Re}(B)
	\end{array} \right) \in \mathfrak{sp}(l,\mathbb{R}).
\end{equation}
It is an isomorphism of Lie algebras since $A\in \mathfrak{u}(l)$ and B are symmetric. The function above is known as the Cayley transform. Furthermore, 
\begin{equation}
	\mathfrak{u}_\Theta \cong \mathfrak{k} \cong \mathfrak{u}(l), \qquad \sqrt{-1}\mathfrak{m} \cong \mathfrak{s},
\end{equation}
where $\mathfrak{sp}(l) = \mathfrak{k} + \mathfrak{s}$ is the Cartan decomposition indicated in \eqref{kc} and \eqref{sc}. Hence, the dual symmetric Lie algebra is isomorphic to $(\mathfrak{sp}(l,\mathbb{R}), \mathfrak{k}, \theta)$, with $\theta$ being the Cartan involution $\theta (X)=\Omega X \Omega^{-1}$, where $\Omega$ is as in \eqref{omega}.

The function defined in \eqref{Cayley} transforms $H_\Theta \in \mathfrak{sp}(l) $ in 
\begin{equation}
	\hat{H}_\Theta := \dfrac{\sqrt{-1}\pi}{2} \left( \begin{array}{cc}
		0& -I_{l\times l} \\ I_{l\times l} &0
	\end{array}\right) \in \mathfrak{sp}(l,\mathbb{R}).
\end{equation}

The dual symmetric space $\mathrm{Sp}(l)^*/U_\Theta \cong \mathrm{Sp}(l,\mathbb{R}) / \mathrm{U}(l)$ is diffeomorphic to $\mathcal{S} = \exp \sqrt{-1}\mathfrak{m} \cdot H_\Theta \cong \exp \mathfrak{s}\cdot \hat{H}_\Theta$. To prove Proposition \ref{lemfibinterS}, let us calculate the intersection $\mathcal{S} \cap (H_\Theta + \mathfrak{n}_\Theta^+)$. Initially, note that
\begin{equation}	\label{Ainter}
	H_\Theta \in	(H_\Theta + \mathfrak{n}_\Theta^+ ) \cap \mathcal{S} \subset \mathfrak{sp}(l)^* \cap (H_\Theta + \mathfrak{n}_\Theta^+).
\end{equation}
To calculate the intersection $\mathfrak{sp}(l)^* \cap (H_\Theta + \mathfrak{n}_\Theta ^+)$, we describe the elements of the fiber passing through the origin $H_\Theta$:
\begin{equation}
	\left\lbrace  \left(  \begin{array}{cc}
		\frac{\sqrt{-1}\pi}{2}I_{l \times l} & Z \\ 0 & -\frac{\sqrt{-1}\pi}{2}I_{l\times l}
	\end{array} \right) : Z \in M_{l\times l} (\mathbb{C}) \right\rbrace .
\end{equation}
with $Z\in \mathfrak{n}_\Theta^+$. These matrices belong to  $\mathfrak{sp}(l)^*$ if, and only if, they belong to the set in \eqref{spl*}, i.e., $Z$ is a complex symmetric matrix and $-\bar{Z}=0$. It means that the intersection $\mathfrak{sp}(l)^* \cap (H_\Theta + \mathfrak{n}_\Theta^+) \subset \{ H_\Theta \}$. According to \eqref{Ainter} we conclude that
\begin{equation}	
	\mathfrak{sp}(l)^* \cap (H_\Theta + \mathfrak{n}_\Theta^+) = \{ H_\Theta \} .
\end{equation}

Thus, $H_\Theta \in (H_\Theta + \mathfrak{n}_\Theta^+) \cap \mathcal{S} \subset (H_\Theta + \mathfrak{n}_\Theta^+) \cap \mathfrak{sp}(l)^* = \{ H_\Theta \}$ and 
\begin{equation}	\label{interfibC}
	(H_\Theta + \mathfrak{n}_\Theta^+) \cap \mathcal{S} = \{ H_\Theta \} .
\end{equation}

Finally, the element $H_\Theta$ is the linear combination 
\begin{gather*}
2\pi (l+1) \sqrt{-1} [ H_{\lambda_1 - \lambda_2} + 2 H_{ \lambda_2 - \lambda_3} + \cdots + (l-1)H_{\lambda_{l-1} - \lambda_l} ]\\ + l(l+1)\pi \sqrt{-1}H_{2\lambda_l}\end{gather*} and $$ \mathcal{K}(\sqrt{-1}H_{2\lambda_l} , H_\Theta) = -4\pi(l+1). $$

\subsection{Case $D_l$}\label{chpt5}
Let $G$ be the Lie group $\mathrm{SO}(2l, \mathbb{C})$ and $U$ be the compact Lie subgroup $\mathrm{SO}(2l)$, with Lie algebras \[\mathfrak{g}:=\mathfrak{so}(2l,\mathbb{C})  = \left\lbrace \left( \begin{array}{cc}
	a&b\\-b^T&d
\end{array} \right) \in \mathfrak{gl}(2l,\mathbb{C}) :  a_{l\times l }, d_{l\times l}\, \text{skew-symmetric} \right\rbrace \] and $\mathfrak{u}:=\mathfrak{so}(2l)$ respectively. The corresponding symmetric flag manifold is the space of orthogonal complex structures on $\mathbb{R}^{2l}$
\begin{equation*}
	\mathbb{F}_\Theta=	\mathrm{SO}(2l)/\mathrm{U}(l)
\end{equation*}
with $\Theta = \{ \lambda_1 - \lambda_2 , \cdots , \lambda_{l-1}- \lambda_{l} \}$ and $\Sigma = \Theta \cup \{ \lambda_{l-1} + \lambda_l \}$ where each $\lambda_{i}$ was defined in \eqref{lambai}.

As $\mathfrak{so}(2l,\mathbb{C})$ is defined over an algebraically closed field if two not degenerated quadratic forms are equivalents, then the algebras of anti-symmetric matrices for the quadratic forms are both isomorphic.

Let 
\begin{equation}	\label{Ff}
	F=\left( \begin{array}{cc}
		0&I_l\\I_l&0
	\end{array} \right) ,\qquad f=\dfrac{1}{\sqrt{2}}\left( \begin{array}{cc}
		\sqrt{-1}I_l & I_l\\I_l&\sqrt{-1}I_l
	\end{array} \right) .
\end{equation}
Then, \begin{equation*}
	F= f^T I_{2l} f.
\end{equation*} Hence, the Lie algebra is isomorphic to 
\begin{equation}	\label{isoDL}
	\mathfrak{g}_1 = \{ A\in \mathfrak{sl}(2l, \mathbb{C}): AF+FA^T=0 \} 
\end{equation}

Thus, $A\in \mathfrak{so}(2l,\mathbb{C})$ if, and only if, $fAf^{-1}\in \mathfrak{g}_1$. In other words $f\mathfrak{so}(2l,\mathbb{C})f^{-1} = \mathfrak{g}_1$ and the algebras $\mathfrak{so}(2l,\mathbb{C})$ and $\mathfrak{g}_1$ are isomorphic. To make easier the computations, we will use the Lie algebra $\mathfrak{g}_1$ instead of $\mathfrak{so}(2l)$, denoting by $G_1$ the connected Lie group with Lie algebra $\mathfrak{g}_1$. 

A Cartan subalgebra of $\mathfrak{g}_1$ is given by
\begin{equation}
	\mathfrak{h}= \left\lbrace \left( \begin{array}{cc}
		\Lambda&0\\0&-\Lambda
	\end{array} \right) : \Lambda_{l\times l} \; \mathrm{diagonal}\; \mathrm{matrix} \right\rbrace .
\end{equation}

An element in the Cartan subalgebra $\sqrt{-1}\mathfrak{h}_\mathbb{R}$ of the compact real form $f \mathfrak{u} f^{-1}= f \mathfrak{so}(2l) f^{-1} $ that defines the symmetric flag manifold $\mathbb{F}_\Theta$ is
\begin{equation*}
	\hat{H}_\Theta :=\dfrac{\pi}{2}\sqrt{-1} \left( \begin{array}{cc}
		I_{l\times l}&0\\0& -I_{l\times l} \end{array} \right) \in \sqrt{-1} \mathfrak{h}_\mathbb{R} \subset \mathfrak{g}_1 
  \,\,\, \mathrm{  and  }\,\,\,  g_\Theta :=\exp(\hat{H}_\Theta ) = \sqrt{-1} \left( \begin{array}{cc}
	I_{l\times l}&0\\0&-I_{l\times l}
	\end{array} \right) .
\end{equation*}

Since $g_\Theta^2 = -Id_{2l} \in Z(G_1)$, the inner automorphism of $G_1$ given by $\sigma = C_{g_\Theta }=g_\Theta Ag_\Theta^{-1}$ is an involution. 

Let us calculate the fixed point set, which is the centralizer of $g_\Theta$ in $G_1$. Since $H_\Theta=\frac{\pi}{2}g_\Theta$, then $ Z_{G_1} (H_\Theta)=Z_{G_1} (g_\Theta)$. Indeed, 
\begin{align}
	Z_{G_1} (g_\Theta) =	& \{ A\in G_1 : C_{g_\Theta} (A)=A \} = \{ A\in G_1 : g_\Theta Ag_\Theta ^{-1}=A \} = \{ A\in G_1 : H_\Theta A = \dfrac{\pi}{2} Ag_\Theta  \} \nonumber \\
	=& \left\{ \left( \begin{array}{cc}
		x_{l\times l} & y\\z&w_{l\times l}
	\end{array} \right)  \in G_1 
H_\Theta A = \dfrac{\pi}{2}  \sqrt{-1} \left( \begin{array}{cc}
x & -y \\ z & -w
\end{array} \right)  \right\} \nonumber \\ 
	=& \left\{ A= \left( \begin{array}{cc}
		x_{l\times l} & y\\z&w_{l\times l}
	\end{array} \right)  \in G_1 :H_\Theta A  = \left( \begin{array}{cc}
		\dfrac{\sqrt{-1}\pi}{2}x & -\dfrac{\sqrt{-1} \pi}{2}y \\ \dfrac{\sqrt{-1}\pi}{2}z & -\dfrac{\sqrt{-1}\pi}{2}w
	\end{array} \right)  \right\} \nonumber \\ 
	=&
\left\{ A= \left( \begin{array}{cc}
	x_{l\times l} & y\\z&w_{l\times l}
\end{array} \right)  \in G_1 :H_\Theta A  =AH_\Theta  \right\}= Z_{G_1}(H_\Theta ). \label{uao1} 
\end{align}

On the other hand, the set of fixed points of $\sigma = $Ad$(g_\Theta)$ (which coincides with \ref{adjunta} in $\mathfrak{g}_1$ is
\begin{align}
	& \left\lbrace A= \left( \begin{array}{cc}
		a&b\\c&-a^T 
	\end{array} \right) \in \mathfrak{g}_1 : a,b,c,d \in M(l,\mathbb{C}) , g_\Theta A g_\Theta^{-1}=A \right\rbrace \nonumber \\
	=& \left\lbrace A= \left( \begin{array}{cc}
		a&b\\c&-a^T	\end{array} \right) \in \mathfrak{g}_1 : b=0=c	\right\rbrace = \left\{  \left( \begin{array}{cc}
		a&0\\0&-a^T
	\end{array} \right) \in \mathfrak{sl}(2l,\mathbb{C}) \right\} \cong \mathfrak{gl}(l,\mathbb{C})	      \label{uao2}. 
\end{align}

Using \eqref{uao1} and \eqref{uao2}, one can see that 
\begin{equation*}
	G/Z_\Theta  \cong \mathrm{SO}(2l, \mathbb{C})/\mathrm{Gl}(l,\mathbb{C}) 
\end{equation*}
is a symmetric space. Its corresponding symmetric pair is $(\mathfrak{so}(2l,\mathbb{C}), \mathfrak{sl}(l,\mathbb{C})\oplus \mathbb{C}^* )$. Hence, this symmetric Lie algebra is in Table \ref{Table2}. Theorem \ref{orbitafibradoSM} implies that the cotangent bundle 
\begin{equation*}
	G/Z_\Theta  = G\cdot H_\Theta = T^*(U/U_\Theta ) 
\end{equation*}
is a symmetric space. In order to find out the set $U_\Theta $, since $U=\mathrm{SO}(2l)$, let us use the matrix $H_\Theta = f^{-1} \hat{H}_\Theta f \in \mathfrak{so}(2l)$. This matrix is
\begin{equation}	\label{hd}
	H_\Theta := \dfrac{\pi}{2} \left(  \begin{array}{cc}
		0&Id_l\\ -Id_l &0
	\end{array} \right) \in \mathfrak{so}(2l,\mathbb{R}) \subset \mathfrak{so} (2l,\mathbb{C}).
\end{equation}
Then,
\begin{align}
	U_\Theta   =& \left\lbrace u=
	\left( \begin{array}{cc}
		a&b\\c&d
	\end{array} \right) \in \mathrm{SO}(2l): a,b \in M_{l\times l}(\mathbb{R}) , uH_\Theta=H_\Theta u
	\right\rbrace \nonumber \\
	=& \left\lbrace 
	\left( \begin{array}{cc}
		a_{l\times l}&-b\\b&a
	\end{array} \right) \in \mathrm{Sl}(2l,\mathbb{R}) : aa^T + bb^T =I_{l\times l},ba^T=ab^T
	\right\rbrace \nonumber \\
	\cong & \{ a+ib \in \mathrm{Gl}(l,\mathbb{C}) : (a+ib)(a+ib)^* = Id \} = \mathrm{U}(l)\nonumber 
\end{align}

Hence, the cotangent bundle of $\mathrm{SO}(2l)/\mathrm{U}(l)$:
\begin{equation}
	T^*(\mathrm{SO}(2l)/\mathrm{U}(l))\cong \mathrm{SO}_\mathbb{C} (2l)/ \mathrm{Gl}(l,\mathbb{C})
\end{equation}
is also a symmetric space.

The tangent space of the cotangent bundle of $\mathbb{F}_\Theta$ at the origin is isomorphic to 
\begin{eqnarray}
	\mathfrak{m}_G =& \mathfrak{n}_\Theta^- + \mathfrak{n}_\Theta^+ \nonumber \\
	= &	\mathfrak{m} + \sqrt{-1} \mathfrak{m}  \nonumber \\
 	\cong & \left\lbrace \left( \begin{array}{cc}
	0&A_{l\times l}\\B_{l\times l}&0
\end{array} \right) :A,B \; \mathrm{anti-symmetric}\; \mathrm{complex} \right\rbrace \label{mgd} 
\end{eqnarray}
where $\mathfrak{n}_\Theta^\pm $ are isomorphic to the subalgebras of matrices in \ref{mgd} with $X=0$ and $Y=0$, respectively. The tangent space of the symmetric flag manifold at the origin $b_\Theta $ is isomorphic to 
\begin{eqnarray}
	\mathfrak{m}\cong & \sum_{\Pi^+ \setminus \langle \Theta \rangle ^+} \mathfrak{u}_\alpha \\
	=& \left\lbrace \left( \begin{array}{cc}
		A&B\\B&-A
	\end{array} \right) : A,B\in \mathfrak{so}(l)  \right\rbrace . \label{tercerm}
\end{eqnarray}

Another real form of $\mathfrak{so}(2l,\mathbb{C})$ is $\mathfrak{g}_0 := \mathfrak{so}^* (2l):= \mathfrak{sl}(l,\mathbb{H}) \cap \mathfrak{so}(2l,\mathbb{C})$ and its elements are $2l \times 2l$ complex matrices of the form
\begin{equation}\left( 
	\begin{array}{cc}
		Z_1&Z_2\\-\bar{Z}_2&\bar{Z_1}
	\end{array}\right) , \qquad Z_1^T=-Z_1,\; \bar{Z_2}^T=Z_2,
\end{equation}
where $Z_1$ and $Z_2$ are $l\times l$ complex matrices.  

The involution $\sigma $ in $\mathfrak{g}_0$ coincides with the Cartan involution. Then, the Lie algebra $\mathfrak{g}_0$ has the Cartan decomposition  $\mathfrak{g}_0= \mathfrak{k} + \mathfrak{s}$ with $\mathfrak{k}= \mathfrak{u}_\Theta \cong \mathfrak{u}(l)$ is the Lie algebra of $U_\Theta $ and $\mathfrak{s}=\sqrt{-1}\mathfrak{m}$.

Finally, to prove Proposition \ref{lemfibinterS}, let us study the intersection of the submanifold $\mathcal{S}= \exp \mathfrak{s} \cdot H_\Theta$ with the fiber passing through the origin. The dual symmetric space $\mathrm{SO}(2l)^*/\mathrm{S}(\mathrm{U}(p)\times \mathrm{U}(q))$ is diffeomorphic to $\mathcal{S} $ and the dual symmetric Lie algebra of $(\mathfrak{so}(2l)^*, \mathfrak{u}_\Theta,\sigma )$, with $\sigma =$Ad(exp$(H_\Theta))$ \eqref{hd} being the function

\begin{equation}
	\sigma : \left( \begin{array}{cc}
		A&B_{l\times l}\\ -B^T&D
	\end{array} \right) \in \mathfrak{so}(2l)^* \mapsto \left( \begin{array}{cc}
		D&B^T\\-B&A
	\end{array} \right) \in \mathfrak{so}(2l)^*,
\end{equation}
is \begin{equation}
	\mathfrak{so}(2l)^* = \mathfrak{u}(l) + \sqrt{-1}\mathfrak{m},
\end{equation}
with $\mathfrak{m}$ as in \eqref{tercerm} and $\mathfrak{u}_\Theta$ as the subalgebra isomorphic to $\mathfrak{u}(l)$:
\begin{equation}
	\left\{ \left( \begin{array}{cc}
		Z_1&Z_2\\-Z_2&Z_1
	\end{array} \right) : Z_1 \in \mathfrak{so}(l), Z_2 \; \mathrm{symmetric}\; \mathrm{matrix} \right\} = \mathfrak{so}^*(2l) \cap \mathfrak{so}(2l)
\end{equation}

Then, 
\begin{equation*}
	\sqrt{-1}\mathfrak{m} = \left\{ \left( \begin{array}{cc}
		\sqrt{-1}A&\sqrt{-1}B\\ \sqrt{-1}B &-\sqrt{-1}A
	\end{array} \right) : A,B \in \mathfrak{so}(l) \right\} = \mathfrak{so}^*(2n) \cap \sqrt{-1}\mathfrak{so}(2l).
\end{equation*}

Since $H_\Theta \in \mathfrak{so}^*(2l) = \mathfrak{sl}(l,\mathbb{H}) \cap \mathfrak{so}(2l, \mathbb{C})$, then $\sigma$ is an automorphism of $\mathfrak{so}^*(2l)$ and it coincides with $\sigma ^*$. Hence, the dual symmetric Lie algebra is 
\begin{eqnarray*}
	\mathfrak{so}^*(2l) =& \mathfrak{so}(2l)^\sigma + \sqrt{-1}\mathfrak{m} \label{DSD} \\
	=& \left\{ \left( \begin{array}{cc}
		Z_1&Z_2\\-\bar{Z}_2&\bar{Z}_1
	\end{array} \right) : Z_1 \in \mathfrak{so}(l,\mathbb{C}), \bar{Z}_2^T=Z_2 \right\}.	\label{sodual}
\end{eqnarray*}

Let us now consider the isomorphism of $\mathfrak{so}(2n,\mathbb{C})$ with the algebra in \ref{isoDL}. The fiber, via this isomorphism, is the affine vector space $\widetilde{H_\Theta + \mathfrak{n}_\Theta^+}$ defined by
\begin{equation*}
\hat{H}_\Theta + f\mathfrak{n}_\Theta^+ f^{-1} =	\left\lbrace \left(  \begin{array}{cc}
		\frac{\sqrt{-1}\pi}{2}I_l &Y\\0&-\frac{\sqrt{-1}\pi}{2}I_l
	\end{array} \right) : Y\in \mathfrak{so}(l,\mathbb{C})) \right\rbrace .
\end{equation*}

Since $f \mathfrak{so}(2n,\mathbb{C}) f^{-1} = \mathfrak{g}_1$, with $f$ as in \eqref{Ff} and $\mathfrak{g}_1$ in \eqref{isoDL}, we can back to the original fiber, calculating $f^{-1}Mf$, for every $M\in \widetilde{H_\Theta + \mathfrak{n}_\Theta^+}$, where
\begin{equation*}
	f^{-1}= \left( \begin{array}{cc}
		-\frac{\sqrt{-1}}{\sqrt{2}} I_l & \frac{1}{\sqrt{2}}I_l\\ \frac{1}{\sqrt{2}} I_l & - \frac{\sqrt{-1}}{\sqrt{2}} I_l
	\end{array} \right) .
\end{equation*}
Then, $	H_\Theta + \mathfrak{n}_\Theta^+ =$
\begin{equation} \left\lbrace  f^{-1} \left( \begin{array}{cc}
		\frac{\sqrt{-1}\pi}{2}I_l & Y\\0& - \frac{\sqrt{-1}\pi}{2}I_l
	\end{array} \right)  f = \left( \begin{array}{cc}
		-\frac{\sqrt{-1}}{2}Y & \frac{\pi}{2}I_l + \frac{1}{2}Y \\ -\frac{\pi}{2}I_l + \frac{1}{2}Y & \frac{\sqrt{-1}}{2}Y
	\end{array} \right)  : Y\in \mathfrak{so}(l,\mathbb{C})\right\rbrace .
\end{equation}
Thus, the elements of the fiber have the form
\begin{equation}	\label{fiberrealim}
	\left(\begin{array}{cc}
		\frac{1}{2} \mathrm{Im}(Y) & \frac{\pi}{2}I_l + \frac{1}{2}\mathrm{Re}(Y) \\ - \frac{\pi}{2} I_l + \frac{1}{2}\mathrm{Re}(Y) & -\frac{1}{2}\mathrm{Im}(Y)
	\end{array}\right) + \sqrt{-1} \left( \begin{array}{cc}
		\frac{-1}{2}\mathrm{Re}(Y) & \frac{1}{2}\mathrm{Im}(Y) \\ \frac{1}{2} \mathrm{Im}(Y) & \frac{1}{2}\mathrm{Re}(Y)
	\end{array} \right) ,
\end{equation}
with $Y\in \mathfrak{so}(l,\mathbb{C}))$. Furthermore, the dual symmetric Lie algebra $\mathfrak{so}(2l)^*$ is the set of matrices 
\begin{equation}	\label{dualrealim}
	\left( \begin{array}{cc}
		A+\sqrt{-1}C & B+\sqrt{-1}D \\ -B+\sqrt{-1}D & A-\sqrt{-1}C
	\end{array} \right) = \left( \begin{array}{cc}
		A&B\\-B&A
	\end{array} \right) + \sqrt{-1} \left( \begin{array}{cc}
		C&D\\D&-C
	\end{array} \right),
\end{equation}with $A,C,D \in \mathfrak{so}(l,\mathbb{R})$ and $B$ being a real symmetric matrix.

Looking at the real part of matrices in \eqref{fiberrealim} and comparing it with the real part of matrices in \eqref{dualrealim}, we can conclude that a matrix in the fiber belongs to $\mathfrak{so}(2l)^*$ if, and only if, Im$(Y)$ and Re$(Y)$ are equal to the null matrix, and then $Y=0$. Hence
\begin{equation} \label{interfibD}
	(	H_\Theta + \mathfrak{n}_\Theta^+ ) \cap \mathfrak{so}(2l)^* = \{ H_\Theta \} .
\end{equation}

The element $\hat{H}_\Theta$ is the linear combination \begin{gather*}
    2\pi (l-1)\sqrt{-1} [ H_{\lambda_1 - \lambda_2} + 2 H_{\lambda_2 - \lambda_3} +\cdots + (l-2)H_{\lambda_{l-2} - \lambda_{l-1}} ] \\+ (l-1) \pi \sqrt{-1} [(l-2)H_{\lambda_{l-1} - \lambda_{l}} + lH_{\lambda_{l-1} + \lambda_{l}}]
\end{gather*} and $$\mathcal{K} ( \sqrt{-1} H_{\lambda_{l-1} + \lambda_l} , \hat{H}_\Theta) = -4\pi (l-1).$$

\section*{Acknowledgements}
The São Paulo Research Foundation (FAPESP) supports L. F. C.  grants 2022/09603-9, 2023/14316-1, and partially supports L. G. grants 2023/13131-8, 2021/04065-6, and L. S. M, grant 2018/13481-0. C. G. was supported by CAPES Ph.D. Scholarship.

The authors gladly acknowledge the anonymous referees for the careful read, leading to useful suggestions that improved the quality of the paper.

	\bibliographystyle{alpha}
	
	\bibliography{main}

\end{document}